\documentclass{article}

\usepackage{amsmath,amsfonts,amsthm,amssymb,amscd,color,xcolor,mathrsfs,verbatim,microtype}
\usepackage{graphicx,eurosym}
\usepackage{hyperref}
\usepackage{mathtools}
\usepackage{bm}

\usepackage[applemac]{inputenc}

\usepackage[cyr]{aeguill}

\colorlet{darkblue}{blue!50!black}

\hypersetup{
    colorlinks,%
    citecolor=blue,%
    filecolor=red,%
    linkcolor=darkblue,%
    urlcolor=blue,%
    pdfnewwindow=true,%
    pdfstartview={FitH}
}

\usepackage{graphicx,amscd,mathrsfs,wrapfig,mathrsfs,lipsum}
\usepackage{eufrak}
\usepackage{float}
\usepackage{tikz}
\usepackage{multicol}
\usepackage{caption}
\usetikzlibrary{arrows}
\usepackage{capt-of}

\colorlet{darkblue}{blue!50!black}

\newcommand{\e}{\varepsilon}

\newcommand{\ppP}{{\mathsf P}}
\newcommand{\qqQ}{{\mathsf Q}}

\newcommand{\C}{{\mathbb C}}

\DeclareMathOperator{\lip}{Lip}
\DeclareMathOperator{\loc}{loc}

\DeclareMathOperator{\realpart}{Re}

\newcommand{\R}{{\mathbb R}}

\newcommand{\pP}{{\mathbb P}}
\newcommand{\I}{{\mathbb I}}

\newcommand{\ty}{\infty}

\newcommand{\aA}{{\cal A}}

\newcommand{\FF}{{\cal F}}

\newcommand{\bb}{{\mathfrak{B}}}

\newcommand{\RR}{{\cal R}}

\newcommand{\dd}{{\textup d}}

\newcommand*\mcap{\mathbin{\mathpalette\mcapinn\relax}}
\newcommand*\mcapinn[2]{\vcenter{\hbox{$\mathsurround=0pt
  \ifx\displaystyle#1\textstyle\else#1\fi\bigcap$}}}

\newcommand*\mcup{\mathbin{\mathpalette\mcupinn\relax}}
\newcommand*\mcupinn[2]{\vcenter{\hbox{$\mathsurround=0pt
  \ifx\displaystyle#1\textstyle\else#1\fi\bigcup$}}}

\theoremstyle{plain}
\newtheorem*{maintheorem}{Main Theorem}

\newtheorem*{lemma*}{Lemma}
\newtheorem{theorem}{Theorem}[section]
\newtheorem{lemma}[theorem]{Lemma}
\newtheorem{proposition}[theorem]{Proposition}
\newtheorem{corollary}[theorem]{Corollary}
\theoremstyle{definition}
\newtheorem{definition}[theorem]{Definition}

\theoremstyle{remark}

\numberwithin{equation}{section}

\begin{document}

\author{Vahagn~Nersesyan\,\footnote{NYU-ECNU Institute of Mathematical Sciences at NYU Shanghai, 3663 Zhongshan Road North, Shanghai, 200062, China, e-mail: \href{mailto:vahagn.nersesyan@nyu.edu}{Vahagn.Nersesyan@nyu.edu}}     \and  Meng~Zhao\,\footnote{School of Mathematical Sciences, Shanghai Jiao Tong University,  Shanghai, 200240, China, e-mail: \href{mailto:mathematics_zm@sjtu.edu.cn}{mathematics$\_$zm@sjtu.edu.cn}}
}

 \date{\today}

\title{Exponential mixing for the white-forced complex Ginzburg--Landau equation in the whole space}
\date{\today}
\maketitle

\begin{abstract}
In the last two decades, there has been a significant progress in the understanding of ergodic properties of white-forced dissipative PDEs. The~previous studies mostly focus on equations posed on bounded domains since they rely on different compactness properties and the discreteness of the spectrum of the Laplacian.  In~the present paper, we consider the damped complex Ginzburg--Landau equation on the real line driven by a white-in-time noise. Under the assumption that the noise is sufficiently non-degenerate, we establish the uniqueness of stationary measure and exponential mixing in the dual-Lipschitz metric. The proof is based on coupling techniques combined with a generalization of  Foia\c{s}--Prodi estimate to the case of the real line and special space-time weighted estimates which help to handle the behavior of solutions at infinity.

\medskip
\noindent
{\bf AMS subject classifications:}      35Q56, 35R60, 37A25, 37L40, 60H15

\medskip
\noindent
{\bf Keywords:} Damped complex Ginzburg--Landau equation,    exponential mixing,   Foia\c{s}--Prodi estimate, stabilization, weighted growth estimates

\end{abstract}

   \tableofcontents

\setcounter{section}{-1}

\section{Introduction}
\label{S:0} 

  The ergodicity of randomly forced partial differential equations (PDEs) has been extensively studied in the literature, especially in the case of equations posed on bounded domains or manifolds;  see the papers \cite{FM95,KS02,EMS,HM06} for the first results and the book~\cite{KS2012} and reviews \cite{FF-2008,D13,KS17} for~subsequent developments. This paper is concerned with the ergodic behavior of the damped complex Ginzburg--Landau (CGL) equation on the real line driven by a white-in-time force:
	\begin{equation}\label{1}
	\begin{cases}
	\partial_t u+au-\nu\partial_{xx}u+\alpha|u|^qu=h(x)+\eta(t,x), \quad x\in \R,\\
			u|_{t=0}=u_0.
	\end{cases}
	\end{equation}
	Here the physical parameters $ a,\nu,\alpha,q$ are  such that
	\[\nu=\nu_1+i\nu_2,\qquad \alpha=\alpha_1+i\alpha_2\]
	with $\nu_1>0$, $\alpha_1\ge 0$,  $\nu_2,\alpha_2\in\mathbb{R}$ and  $a>0$, $q\in (0,2)$. It is worth noting that we do not impose any restriction on the size of the damping parameter $a$ and the real viscosity $\nu_1$.
	  Let $H:=L^2(\mathbb{R};\C)$ be the usual   Lebesgue space of functions $u:\R \to \C$  endowed with the scalar product
	\[
	\langle u,v\rangle:=\realpart \int_{\mathbb{R}} u(x)\bar{v}(x)\dd x
	\]
	and the norm $\|u\|^2:=\langle u,u\rangle$. The~driving force consists of two parts: $h$ is a given deterministic function belonging to the space $H$, while $\eta$ is a noise of the form
	\begin{align}\label{2}\eta(t,x):=\partial_t\sum_{j=1}^{\infty}b_j\beta_j(t) e_j(x),\end{align}
	where $\{b_j\}$ are real numbers such that
	\begin{equation}\label{E:0.3}
		 \bb_1:=\sum_{j=1}^\ty b_j^2<\ty,
	\end{equation}
  $\{e_j\}_{j=1}^{\infty}$ is an orthonormal basis in  $H$, and $\{\beta_j\}$  are independent standard real Brownian motions defined on a probability space   $(\Omega,\FF,\FF_t,\pP)$ satisfying the usual conditions (e.g., see Definition~2.25 in~\cite{KS91}). 
 In this setting, the stochastic CGL equation \eqref{1} is globally well-posed and defines a Markov process in $H$.  	 In this paper, we prove the following result.
	\begin{maintheorem}
		Assume that $h\varphi\in H$ and
		\begin{gather}
					\bb_2:=\sum_{j=1}^{\infty}b_j^2\|\varphi e_j\|^2<\infty,\qquad  \bb_3:=\sum_{j=1}^{\infty}b_j^2\|\partial_{x} e_j\|^2<\infty,\label{3}\\
					\label{E:0.3}	  \sum_{j=1}^{\infty}|\langle h,e_j\rangle| \|e_j\|_{H^1}<\infty,
		\end{gather}
		where $\varphi(x):=\log(x^2+2)$.   		Then, there exists an integer $N\ge1$ such that the stochastic CGL equation~\eqref{1} admits a unique stationary probability measure $\mu$, provided that  
		\begin{equation}\label{E:bN}
	 	b_j\neq 0\quad \text{for $1\le j\le N$}.
	 \end{equation} Moreover, there are constants $C,\kappa>0$ such that 
		\[\left|\mathbb{E}(f(u(t)))-\int_{H}f(u)\mu(\dd u)\right|\le C\|f\|_{\lip}e^{-\kappa t}\left(1+ \|u_0\|^2\right), \quad t\ge0\]
		for any bounded Lipschitz-continuous function $f:H\to \R$  and any initial data~$u_0\in H$.
	\end{maintheorem}
     To the best of our knowledge, this is the first result that establishes the uniqueness of stationary measure and exponential mixing for PDEs on unbounded domains driven by an additive white-in-time noise of the form~\eqref{2}.~The existence of a stationary measure for similar stochastic PDEs has been previously studied by several authors; e.g.,  see ~\cite{BL06}
for the case of the Navier--Stokes~(NS) system and~\cite{Kim-2006, EKZ-17} for nonlinear Schr\"odinger equations with damping.

     There are only a few results proving the uniqueness of stationary measure and mixing for randomly forced PDEs on unbounded domains. Most of them are in the case of Burgers-type equations perturbed by a space-time homogeneous noise. The paper~\cite{BCK12} considers the inviscid Burgers equation on the real line with a space-time homogeneous Poisson point process. The proof is based on Lagrangian methods and first/last passage percolation theory. Later, \cite{BL19} extends the result to the viscous case by considering as a perturbation a space-time homogeneous random kick force. More recently, the paper~\cite{DGR-21} uses the Hopf--Cole transform, the comparison principle, and $L^1$-contraction for the Burgers equation to prove uniqueness and mixing in the case of a space-time Gaussian noise.  
       These papers employ properties specific to the Burgers equation, and
 their methods do not apply to CGL or NS-type systems with a space-time homogeneous noise. For~existence results in the case of real/complex Ginzburg--Landau equations, we refer to the papers~\cite{EH01, ROU02}.

       In a non-homogeneous setting,  the uniqueness and exponential mixing for the NS system on unbounded domains is established in~\cite{N22} in the case of a bounded noise. The proof is carried out by developing a controllability approach combined with the asymptotic compactness of the dynamics.
 
   The proof of Main Theorem relies on the coupling method which has already proved to be very efficient in the case of many equations on bounded domains. More precisely, we develop the version of the method presented in the paper~\cite{S2008} and Section~3.1.2 in the book~\cite{KS2012}.      Extending this approach to unbounded domains presents at least two significant challenges: first, loss of compactness properties, including the compactness in Sobolev embeddings, and second, the spectrum of the Laplace operator is no longer purely discrete.

    One of the crucial components of the method is the Foia\c{s}--Prodi type estimate, which is a stabilization property by an additive finite-dimensional feedback force. It is of utmost importance that the finitely many directions in the feedback are chosen from the Cameron--Martin space of the noise, enabling effective application of the Girsanov theorem. In the current unbounded setting, we design a new form of feedback that does not rely on the eigenfunctions of the Laplacian and is strong enough to stabilize the trajectories over the entire real line. The verification of the Novikov condition is the most technical part of the paper. To tackle this, we introduce a special space-time weight function that allows to estimate appropriately the behavior of solutions at infinity. The use of this weight is inspired by the attractor theory for deterministic parabolic PDEs (see \cite{A1989}), and as far as we know, this is the first time it is used in a stochastic framework.

In this paper, we opt to focus on the one-dimensional CGL for the sake of simplicity. However, the method is quite robust, and with similar arguments, the results extended to higher dimensions under a suitable restriction on the power~$q$. It is noteworthy that our proof uses the local nature of the nonlinearity of the CGL equation and it does not extend directly to the case of the NS system. The~latter will be studied in a forthcoming publication.

    The paper is structured as follows. Section~\ref{maintheorem} provides a detailed construction of the coupling processes and explains how the proof of the exponential mixing reduces to the verification of appropriate recurrence and exponential squeezing properties. Foia\c{s}--Prodi type estimate is derived in Section~\ref{stability} together with some growth estimates for an auxiliary process. The recurrence and exponential squeezing properties are verified in
      Section~\ref{proofoftheorem2}. Finally, in the Appendix, we present the proofs of several technical lemmas and propositions that were utilized in the preceding sections.

 \subsubsection*{Acknowledgement}  
   
The authors are grateful to  Professors Armen Shirikyan and Ya-Guang Wang for stimulating and helpful discussions. The second author was partially supported by National Natural Science Foundation of China under Grant No. 12171317.
  
 \subsubsection*{Notation}

Throughout this paper, we use the following notation.  For a Banach space $X$ endowed with norm  $\|\cdot\|$, we denote:
	
   \smallskip
   \noindent
   $C_b(X)-$the space of bounded continuous functions   $f:X\to \R$ endowed with the norm
	\[\|f\|_{\infty}:=\sup_{x\in X}|f(x)|;\]

   \smallskip
   \noindent
   $\lip(X)-$the space of bounded Lipschitz-continuous functions   $f:X\to \R$ endowed with the norm
	\[\|f\|_{\lip}:=\|f\|_{\infty}+\sup_{\substack{x,y\in X\\ x\neq y}}\frac{|f(x)-f(y)|}{\|x-y\|};\]

   \smallskip
   \noindent
   $B_X(x,r)-$the open ball in $X$ of radius $r$ centered at $x\in X$;

   \smallskip
   \noindent
   $\overline{B}_X(x,r)-$the closure of $B_X(x,r)$;
	
   \smallskip
   \noindent
   $\mathscr{B}(X)-$the Borel $\sigma$-algebra of $X$;
	
   \smallskip
   \noindent
   $\mathscr{P}(X)-$the set of Borel probability measures on $X$. We write
   $$
   (f,\lambda):=\int_{X}f(u)\lambda(\dd u)
   $$
   for  $f\in C_b(X)$ and $\lambda\in\mathscr{P}(X)$. For $\lambda_1,\lambda_2\in \mathscr{P}(X)$, we set 
   \begin{gather*}
      		\|\lambda_1-\lambda_2\|_{\mathcal{L}}^*:=\sup_{\substack{f\in\lip(X)\\\|f\|_{\lip(X)}\le1}}|(f,\lambda_1)-(f,\lambda_2)|,\\
	\|\lambda_1-\lambda_2\|_{\textup {var}}:=\sup_{\Gamma\in \mathscr{B}(X)}|\lambda_1(\Gamma)-\lambda_2(\Gamma)|=\frac{1}{2}\sup_{\substack{f\in C_{b}(X)\\\|f\|_{\infty}\le1}}|(f,\lambda_1)-(f,\lambda_2)|.
   \end{gather*} Let $H:=L^2(\mathbb{R};\C)$ be the usual   Lebesgue space of complex-valued functions on~$\mathbb{R}$ endowed with the scalar product
	\[
	\langle u,v\rangle:=\realpart \int_{\mathbb{R}} u(x)\bar{v}(x)\dd x
	\]
	and the norm $\|u\|^2:=\langle u,u\rangle$. By $H^s:=H^s(\mathbb{R};\C)$, $s\in\mathbb{R}$ we denote the Sobolev space of complex-valued functions on $\mathbb{R}$ with the norm 
	\[
	\|u\|_{H^s}^2:=\int_{\mathbb{R}}(1+|\xi|^2)^s|\mathscr{F} u(\xi)|^2\dd\xi,
	\]
	where $\mathscr{F} u$ is the Fourier transform of $u$.

 The distribution of an  $X$-valued random variable $\zeta$ is denoted by   $\mathscr{D}(\zeta)$. For~any set $\Gamma\subset X$, we write $\I_\Gamma$ for the characteristic function. $a\vee b$ and~$a\wedge b$ are the maximum and minimum of   numbers $a,b\in \R$. Finally, we denote by $C$, $C_a$, $C_\nu$, $\ldots$ unessential positive constants, and the subscripts are used to express the dependence on the parameters. 
   For the sake of simplicity, we shall often use the symbols $\lesssim$, $\lesssim_a$, $\lesssim_\nu $, $\dots$ when an inequality holds up to an unessential multiplicative constant $C$, $C_a$, $C_\nu$, $\ldots$.

  \section{Main result and construction of a mixing extension}\label{maintheorem}

	  The stochastic CGL equation \eqref{1} defines a Markov family $(u_t,\mathbb{P}_u)$ parameterized by the initial condition~$u\in H$. The following standard energy estimate is proved in the Appendix:
	   \begin{align}\label{19}
		\mathbb{E}_u\|u(t)\|^2\le e^{-at}\|u\|^2+C',
		\end{align}
		where $\mathbb{E}_u$ is the expectation with respect to~$\mathbb{P}_u$ and $C':=\frac{1}{a}\left(\frac{\|h\|^2}{a}+\bb_1\right)$.   Let us denote by~$S_t(u,\cdot)$ the flow    issued from $u\in H$ and    by 
	\[
	\mathscr{B}_t:C_b(H)\to C_b(H), \qquad  \mathscr{B}_tf(u):=\int_{H}f(v)P_t(u,\dd v),\qquad f\in C_b(H),
	\]
	\[
	\mathscr{B}^*_t:\mathscr{P}(H)\to \mathscr{P}(H), \qquad \mathscr{B}^*_t\lambda(\Gamma):=\int_{H}P_t(u,\Gamma)\lambda(\dd u),\qquad  \lambda\in \mathscr{P}(H)
	\] the associated Markov operators,
	where $P_t(u,\Gamma):=\mathbb{P}\left\{S_t(u,\cdot)\in\Gamma\right\}$ is the transition function. Recall that a measure is called stationary  for the family $(u_t,\mathbb{P}_u)$, if $\mathscr{B}_t^{*}\mu=\mu$ for any $t>0$. 
	We reformulate   Main Theorem as follows.
	\begin{theorem}\label{theorem1}
		Under the hypotheses of Main Theorem, the family $(u(t),\mathbb{P}_u)$ has a unique stationary measure $\mu\in\mathscr{P}(H)$. Moreover, there are  constants~$C,\kappa>0$ such that 
		\[\|\mathscr{B}_t^*\lambda-\mu\|_{\mathcal{L}}^*\le Ce^{-\kappa t}\left(1+\int_{H}\|u\|^2\lambda(\dd u)\right)\]
		for any $\lambda\in\mathscr{P}(H)$.
	\end{theorem}
    \noindent\textit{Scheme of the proof.} The proof is based on the coupling approach described in detail in Section~3.1.3 of the book~\cite{KS2012} (see also~\cite{S2008}). Let us briefly summarize the main ingredients. Let~$(\bm{u}_t,\mathbb{P}_{\bm{u}})$ be a Markov family in $H\times H$. We~shall say that $(\bm{u}_t,\mathbb{P}_{\bm{u}})$ is an extension of~$(u_t,\mathbb{P}_u)$ if for any $\bm{u}=(u,u')\in H\times H$ the laws under~$\mathbb{P}_{\bm{u}}$ of processes~$\{\Pi_1\bm{u}_t\}_{t\ge0}$ and $\{\Pi_2\bm{u}_t\}_{t\ge0}$ coincide with those of~$\{u_t\}_{t\ge0}$ under $\mathbb{P}_u$ and $\mathbb{P}_{u'}$ respectively, where $\Pi_1$ and $\Pi_2$ denote  the projections from~$H\times H$ to the first and second component. The main idea of the coupling approach is to construct an extension $(\bm{u}_t,\mathbb{P}_{\bm{u}})$   of $(u_t,\mathbb{P}_u)$ possessing a recurrence and exponential squeezing properties, as stated in the following theorem.

    \begin{theorem} \label{theorem2}Under the hypotheses of Main Theorem, there is an extension $(\bm{u}_t,\mathbb{P}_{\bm{u}})$ of   $(u_t,\mathbb{P}_{u})$ and a stopping time $ \sigma$ such that the following properties are satisfied.  
			\begin{enumerate}
				\item {\rm (Recurrence):} There exist constants $\delta,d,C>0$ such that 				\begin{align}\label{14}\mathbb{E}_{\bm{u}}\exp(\delta \tau_d)\le C(1+\|u\|^2+\|u'\|^2)\end{align}
				for any $\bm{u}\in H\times H$, where 
  $\tau_d$ is   the first hitting time of    $\overline{B}_{H}(0,d)\times \overline{B}_{H}(0,d)$, i.e.,
$$
  \tau_d:=\inf\{t\ge0 \,|\,\bm{u}_t\in  \overline{B}_{H}(0,d)\times \overline{B}_{H}(0,d)\}.		
$$
   	\item {\rm (Exponential squeezing):} There are constants $\delta_1,\delta_2,c,c'>0$  such that 
				\begin{align}\label{15}&\|\tilde{u}(t)-\tilde{u}'(t)\|^2\le c e^{-c' t} \|u-u'\|^2\qquad \mbox{ for $0\le t\le \sigma$},\\ \label{16}&\mathbb{P}_{\bm{u}}\{{\sigma}=\infty\}\ge \delta_1,\\\label{17}&\mathbb{E}_{\bm{u}}\left(\I_{\{{\sigma}<\infty\}}\exp(\delta_2 {\sigma})\right)\le c,\\\label{18}&\mathbb{E}_{\bm{u}}\left(\I_{\{{\sigma}<\infty\}}\left(\|\tilde{u}({\sigma})\|^{4}+\|\tilde{u}'({\sigma})\|^{4}\right)\right)\le c
				\end{align}for any     $\bm{u}\in\overline{B}_{H}(0,d)\times \overline{B}_{H}(0,d)$, where $\tilde u_t=\Pi_1\bm{u}_t$ and $\tilde u'_t=\Pi_2\bm{u}_t$. 
 			\end{enumerate}
		\end{theorem}
		A mixing extension is an extension satisfying the above recurrence and exponential squeezing properties.
	 	According to Theorem~3.1.7 in~\cite{KS2012}, existence of a mixing extension implies existence of a unique stationary measure and  exponential mixing
	 	as in Theorem~\ref{theorem1}. Let us explain how the mixing extension is constructed in the case of the stochastic CGL equation \eqref{1}.

  \smallskip
  
    \noindent {\it Construction of the mixing extension $(\bm{u}_t,\mathbb{P}_{\bm{u}})$}. We develop the strategy presented in~\cite{M2014} in the case of the nonlinear wave equation on a bounded domain. For initial points $u,u'\in H$, let $\{u(t)\}_{t\ge 0}$ and $\{u'(t)\}_{t\ge0}$ denote the solutions of the equation \eqref{1} starting from $u$ and $u'$, respectively. We introduce an auxiliary process $\{v(t)\}_{t\ge0}$ which is the solution of the problem
		\begin{align}\label{8}\begin{cases}
				\partial_t v+av-\nu\partial_{xx}v+\alpha|v|^{q}v+\ppP_N\left[\alpha|u|^qu-\alpha|v|^qv-\nu\partial_{xx}(u-v)\right]=h+\eta,\\
				v|_{t=0}=u',
		\end{cases}\end{align}
    where  $\ppP_N$ denotes the orthogonal projection in $H$ onto the space spanned by the family~$\{e_1,\ldots,e_N\}$  with integer $N\ge1$ to be specified later. The finite-dimensional term 
    \begin{equation}\label{EE:FP}
    	\ppP_N\left[\alpha|u|^qu-\alpha|v|^qv-\nu\partial_{xx}(u-v)\right]
    \end{equation}
     is a sort of feedback control that effectively stabilizes the trajectory of \eqref{8} issued from $u'$ to that of \eqref{1} issued from~$u$ on an event of sufficiently large probability. It is worth noting that the form of this term differs from the commonly used feedbacks in the case of the Navier--Stokes system (cf. Section~2.1.8 in~\cite{KS2012}) and nonlinear wave equations (cf. Section~4.1 in~\cite{M2014}). This difference is due to the unboundedness of the underlying domain and the fact that $\{e_j\}$ are not eigenfunctions of the Laplacian.

      Let $ T>0$ be a time parameter that will be chosen later.    We denote by $\lambda_T(u,u')$ and $\lambda'_T(u,u')$   the distributions of   processes $\{v(t)\}_{t\in[0,T]}$ and~$\{u'(t)\}_{t\in[0,T]}$, respectively. Then $\lambda_T(u,u')$ and $\lambda'_T(u,u')$ are probability measures on   $C([0,T];H)$. By Theorem~1.2.28 in~\cite{KS2012}, there exists a maximal coupling $(\mathcal{V}_T(u,u'),\mathcal{V}_T'(u,u'))$ for the pair $(\lambda_T(u,u'),\lambda'_T(u,u'))$ defined on some probability space  $(\tilde \Omega, \tilde \FF, \tilde\pP)$. We denote   by $\{\tilde{v}(t)\}_{t\in[0,T]}$ and $\{\tilde{u}'(t)\}_{t\in[0,T]}$ the   flows of this maximal coupling.  Then, $\tilde{v}$ is the solution of 
		\begin{align}\label{9}\begin{cases}
				\partial_t \tilde{v}+a\tilde{v}-\nu\partial_{xx}\tilde{v}+\alpha|\tilde{v}|^{q}\tilde{v}+\ppP_N\left[\nu\partial_{xx}\tilde{v}-\alpha|\tilde{v}|^{q}\tilde{v}\right]=h+\Lambda,\\
				\tilde{v}|_{t=0}=u',
		\end{cases}\end{align}
		where the process $\{\Lambda(t)\}_{t\in[0,T]}$ satisfies the following property:
	\begin{itemize}
	\item[\hypertarget{(A)}{\bf(A)}] the distribution of $\left\{\int_0^{t}\Lambda(s)\dd s\right\}_{t\in[0,T]}$ is equal to that of the process 
	$$
	\left\{\int_0^{t}\left(\eta(s)-\alpha \ppP_N|u(s)|^{q}u(s)+\nu \ppP_N\partial_{xx}u(s)\right)\dd s\right\}_{t\in[0,T]},
	$$where $\eta$ is defined by \eqref{2}.
\end{itemize} Let the process $\{\tilde{u}(t)\}_{t\in [0,T]}$ be the solution of 
		\begin{align}\label{11}\begin{cases}
				\partial_t \tilde{u}+a\tilde{u}-\nu\partial_{xx}\tilde{u}+\alpha|\tilde{u}|^{q}\tilde{u}+\ppP_N\left[\nu\partial_{xx}\tilde{u}-\alpha|\tilde{u}|^{q}\tilde{u}\right]=h+\Lambda,\\
				\tilde{u}|_{t=0}=u.
		\end{cases}\end{align}
		We claim that $\mathscr{D}(\{\tilde{u}(t)\}_{t\in[0,T]})=\mathscr{D}(\{u(t)\}_{t\in[0,T]})$. Indeed, we derive from~\eqref{1} that 
		\[\partial_t u+au-\nu\partial_{xx}u+\alpha|u|^{q}u+\ppP_N\left[\nu\partial_{xx}u-\alpha|u|^{q}u\right]=h+\bar{\Lambda},\]
		where 
		\[\bar{\Lambda}:=\eta-\alpha \ppP_N|u|^{q}u+\nu \ppP_N\partial_{xx}u.\]
		Due to~\hyperlink{(A)}{(A)}, we have 
		\[\mathscr{D}\left(\left\{\int_0^{t}\Lambda(s)\dd s\right\}_{t\in[0,T]}\right)=\mathscr{D}\left(\left\{\int_0^{t}\bar{\Lambda}(s)\dd s\right\}_{t\in[0,T]}\right),\]
		which implies that $\mathscr{D}(\{\tilde{u}(t)\}_{t\in[0,T]})=\mathscr{D}(\{u(t)\}_{t\in[0,T]})$ as claimed, because of the uniqueness in law for the stochastic PDE \eqref{11}.  		
		
	  We define operators~$\RR$ and~$\RR'$  by 
$$
\RR_t(u,u',\omega)=\tilde u_t,\quad \RR'_t(u,u',\omega)=\tilde u_t'
$$ for  any $u, u'\in H$,  $\omega\in\tilde\Omega$, and $t\in [0,T]$. Then, let $\{(\Omega^k,\FF^k,\pP^k)\}_{k\ge0}$ be a sequence of independent copies of   $(\tilde \Omega, \tilde \FF, \tilde\pP)$, and let $(\Omega,\FF,\pP)$ be the   direct product of~$(\Omega^k,\FF^k,\pP^k)$. For any $\omega=(\omega^1,\omega^2,\ldots)\in\Omega$ and $u,u'\in H$, we set~$\tilde u_0=u$, $\tilde u_0'=u'$,  and 
\begin{gather*}
\tilde u_t(\omega):=\RR_s(\tilde u_{k}(\omega),\tilde u'_{k}(\omega),\omega^k), \quad
\tilde u'_t(\omega):=\RR_s'(\tilde u_{k}(\omega),\tilde u'_{k}(\omega),\omega^k),\\
\bm{u}_t:=(\tilde u_t,\tilde u_t'),
\end{gather*} 
where $t=s+kT$, $s\in [0,T)$. The construction implies that $(\bm{u}_t,\mathbb{P}_{\bm{u}})$  is an extension for  $(u_t,\mathbb{P}_u)$.
 We will show in Section~\ref{proofoftheorem2} that, for an appropriate choice of parameters $N$ and $T$, the process  $(\bm{u}_t,\mathbb{P}_{\bm{u}})$  is a mixing extension for  $(u_t,\mathbb{P}_u)$.  In preparation for that, in Section~\ref{stability}, we lay the groundwork by studying some stability properties of the CGL equation.

	\section{Stability of solutions}\label{stability}

We begin this section by establishing a Foia\c{s}--Prodi type estimate for the~CGL equation on the real line. Then we provide a growth estimate for a process involved in this estimation. Throughout this section and the rest of the paper, we always assume that the conditions of Main Theorem are satisfied.

	\subsection{Foia\c{s}--Prodi type estimate}

  For PDEs posed on bounded domains, Foia\c{s}--Prodi type estimates are well-known and have been extensively discussed in Section~2.1.8 in~\cite{KS2012} as well as in the original work of Foia\c{s} and Prodi~\cite{FP67}. However, as mentioned in the Introduction, when dealing with unbounded domains, significant challenges arise due to the lack of compactness in the Sobolev embeddings and the non-discreteness of the spectrum of the Laplacian. To overcome~these difficulties,  we use the compactness of the embedding $H^{s}(\mathbb{R})\hookrightarrow L^2_{\loc}(\mathbb{R})$,~$s>0$, more specifically, its consequence in the form of a truncated   Poincar\'e inequality. To~formulate the latter, for any~$A>0$,   let us denote by $\chi_A$ any smooth cut-off function~$\chi_A: \mathbb{R}\to [0,1]$   such~that 
	\[\chi_A=\begin{cases}
		1,\qquad x\in[-A/2,A/2],\\
		0,\qquad x\in [-A,A]^\mathsf{c}.
	\end{cases}\]
		\begin{lemma}\label{lemma1}
	Let us fix any $s>0$. For any $\varepsilon, A>0$, there is an integer $N\ge1$ such that 
		\begin{equation}\label{20}
		\|\qqQ_N\chi_A f\|\le\varepsilon\|f\|_{H^{s}} \quad \text{for  $f\in H^{s}(\R,\C)$,}
		\end{equation}
		 where $\qqQ_N:=\textup{I}-\ppP_N$ and $\ppP_N$ is the same orthogonal projection as in \eqref{8}.
	\end{lemma}
	We postpone the proof of this lemma to the Appendix and turn to the  Foia\c{s}--Prodi estimate. Inspired by \cite{A1989}, let us introduce a smooth space-time weight function given by
		\begin{equation}\label{12}
		\psi(t,x):=\varphi(x)\left(1-\exp\left(-\frac{t}{\varphi(x)}\right)\right), \quad (t,x)\in \R^2,
		\end{equation}
		where $\varphi(x):=\log(x^2+2)$. We will use the following properties of this function:
		\begin{itemize}
			\item[\hypertarget{(i)}{\bf(i)}] $0<\psi(t,x)<\varphi(x)$ for $t> 0$ and $\psi(0,x)=0$;
			\item[\hypertarget{(ii)}{\bf(ii)}]  the partial derivatives of order $\ge 1$ of $\psi$  are bounded functions;
			\item[\hypertarget{(iii)}{\bf(iii)}] $\psi(t,x)\to +\ty$ as 
			$t,|x|\to+\infty$.			
					\end{itemize}
		 Let us denote 
	$w(t):=u(t)-v(t),$
	where~$\{u(t)\}_{t\ge 0}$ and $\{v(t)\}_{t\ge0}$ are the solutions of \eqref{1} and \eqref{8}, respectively.

	\begin{proposition}\label{proposition1}
	For any $\varepsilon>0$, there is a time $T>0$ and an integer $N\ge1$ such that 
		\begin{align}\label{22}
			&\|w(t+T)\|^2\le \left(\|w(s+T)\|^2+\frac{|\nu|^2}{a\nu_1}\|\partial_x \ppP_N w(s+T))\|^2\right)\notag\\&\quad\quad \times \exp\left(-a(t-s)+C\varepsilon\int_{s+T}^{t+T}\left(\|u\|_{H^1}^2+\|\psi u\|_{H^1}^2+\|v\|_{H^1}^2+\|\psi v\|_{H^1}^2\right)\dd r\right)
		\end{align}
		for any $t\ge s\ge 0 $, where   $C>0$ is a constant depending on   $a,\nu,\alpha,q$.
	\end{proposition}
	Inequality \eqref{22} is a Foia\c{s}--Prodi type estimate for the CGL equation \eqref{1} and it will play a central role. Combining some estimates derived in the next subsection and in the Appendix, we will show that the integral term on the right-hand side of this inequality grows sublinearly in time on an event of large probability. Therefore, by choosing~$\e$ small enough compared to the damping~$a$, we will derive that $\|w(t+T)\|$ decreases exponentially. This, combined with the Girsanov theorem, will enable us to establish the exponential squeezing property in Theorem~\ref{theorem2}.

	\begin{proof}[Proof of Proposition~\ref{proposition1}] Note that $w$ is a solution of the equation 
        \begin{align}\label{22a}\partial_tw+aw+\qqQ_N[\alpha|u|^{q}u-\alpha |v|^{q}v-\nu\partial_{xx}w]=0.\end{align}
        Taking the scalar product in $H$ of this equation  with $w$ and integrating by parts, we get
       		\begin{align}\label{23}
			\frac{1}{2}\frac{\dd}{\dd t}\|w\|^2+a\|w\|^2+&\nu_1\|\partial_x w\|^2\notag\\&=\langle \qqQ_Nw,\alpha(|v|^{q}v-|u|^{q}u)\rangle+\langle \ppP_Nw,-\nu\partial_{xx}w\rangle.
		\end{align}
		To estimate the right-hand side of this equality, we write 
		\begin{align}\label{24}
			\langle \ppP_N w,-\nu\partial_{xx}w\rangle\le |\nu|\|\partial_x \ppP_Nw\|\|\partial_x w\|\le \frac{\nu_1}{2}\|\partial_x w\|^2+\frac{|\nu|^2}{2\nu_1}\|\partial_x \ppP_N w\|^2
		\end{align}
		and 
		\begin{align}\label{25}
			\langle \qqQ_Nw, \alpha|v|^{q}v-\alpha|u|^{q}u\rangle &=\langle \qqQ_Nw, \alpha\chi_A[|v|^{q}v-|u|^{q}u]\rangle \nonumber \\ &\quad+\langle \qqQ_Nw, \alpha(1-\chi_A)[|v|^{q}v-|u|^{q}u]\rangle\notag\\&=:I_1+I_2.
		\end{align}
		For $I_1$, we claim that 
        \begin{align}\label{25a}I_1\le \varepsilon C_{a,\nu,\alpha,q}\|w\|^2(\|u\|_{H^1}^2+\|v\|_{H^1}^2)+\frac{a\wedge \nu_1}{2}\|w\|_{H^1}^2.\end{align}
        Indeed, if $q\in (1,2)$,  we apply Lemma \ref{lemma1} with $s=2-q$ and arbitrary $\varepsilon, A>0$:
		\begin{align}\label{26}
			I_1&\le |\alpha|\|w\|\|\qqQ_N\chi_A[|v|^{q}v-|u|^{q}u]\|\le \varepsilon |\alpha|\|w\|\||v|^{q}v-|u|^{q}u\|_{H^{2-q}}\notag\\&\le \varepsilon|\alpha|\|w\|\||v|^{q}v-|u|^{q}u\|^{q-1}\||v|^{q}v-|u|^{q}u\|_{H^1}^{2-q},
		\end{align}
		where 
		\begin{align}\label{27}
			\||v|^{q}v-|u|^{q}u\|& \lesssim_{q}\|w(|u|^{q}+|v|^{q})\|\lesssim_{q}\|w\|\left(\|u\|^{q}_{H^1}+\|v\|^{q}_{H^1}\right)
		\end{align}
		and  
		\begin{align}\label{28}
			\|\partial_x(|v|^{q}v-&|u|^{q}u)\| \lesssim_{q}\|w\|_{H^1}\left(\|u\|^{q}_{H^1}+\|v\|^{q}_{H^1}\right).
		\end{align}
		Plugging \eqref{27} and \eqref{28} into \eqref{26}, we see that
		$$
			I_1\lesssim_{q} \varepsilon|\alpha|\|w\|^{q}\|w\|_{H^1}^{2-q}\left(\|u\|^{q}_{H^1}+\|v\|^{q}_{H^1}\right),
		$$
        which leads to \eqref{25a} in view of Young's inequality. For $q\in (0,1]$, the estimate is similar. Indeed, applying Lemma~\ref{lemma1} with $s=1$ and arbitrary $\varepsilon,A>0$ and using the estimates \eqref{27} and  \eqref{28}, we infer that
        \begin{align*}
        I_1&\le \varepsilon C_{\alpha}\|w\|\||v|^qv-|u|^qu\|_{H^1}\\&\le \varepsilon C_{\alpha}\|w\|\|w\|_{H^1}\left(\|u\|_{H^1}^q+\|v\|_{H^1}^q\right)\\&\le \varepsilon C_{a,\nu,\alpha,q}\|w\|^{2}  \left(\|u\|_{H^1}^{2}+\|v\|_{H^1}^{2}\right)+\frac{a\wedge \nu_1}{2}\|w\|^2_{H^1}
        \end{align*}
        as claimed. 
        As for $I_2$, from the property~\hyperlink{(iii)}{(iii)} of the weight function $\psi$ it follows~that
		\begin{align*}
			\frac{1}{\psi^{q}(t,x)}\le \varepsilon\qquad\text{for  $t\ge T,\, |x|\ge \frac{A}{2}$},
		\end{align*}provided that $T,A>0$ are large enough. 
 This implies that 
		\begin{align*}
			I_2&=\langle \qqQ_Nw, \alpha(1-\chi_A)\psi^{-q}[|\psi v|^{q}v-|\psi u|^{q}u]\rangle\notag\\&\le\varepsilon|\alpha|\|w\|\||\psi v|^{q}v-|\psi u|^{q}u\|
		\end{align*}
		with (cf. \eqref{27})   
		$$
			\||\psi v|^{q}v-|\psi u|^{q}u\|\lesssim_{q} \|w\|(\|\psi u\|^{q}_{H^1}+\|\psi v\|^{q}_{H^1}).
		$$
		Hence, 
		\begin{align}\label{32}
			I_2&\lesssim_{q}\varepsilon|\alpha|\|w\|^2\left(\|\psi u\|^{q}_{H^1}+\|\psi v\|^{q}_{H^1}\right)
			\nonumber \\&\le \varepsilon C_{a,\nu,\alpha,q}\|w\|^{2}  \left(\|\psi u\|_{H^1}^{2}+\|\psi v\|_{H^1}^{2}\right)+\frac a2\|w\|^2 .
		\end{align}
		Combining \eqref{23}-\eqref{25a} and \eqref{32}, we obtain 
		\begin{align*}
	\frac{\dd}{\dd t}\|w\|^2+\left(a-C\varepsilon\left(\|u\|^2_{H^1}+\|v\|^2_{H^1}+\|\psi u\|^2_{H^1}+\|\psi v\|^2_{H^1}\right)\right)&\|w\|^2\\\le \frac{|\nu|^2}{\nu_1}&\|\partial_x \ppP_Nw\|^2,
	\end{align*}
		where $C$ depends on $a,\nu,\alpha, q$. An application of Gronwall's inequality, together with the fact that
		\begin{equation}\label{E:PN}
					\ppP_Nw(t)=e^{-a(t-s)}\ppP_Nw(s),
		\end{equation}
          yields \eqref{22} as required. 
	\end{proof}

	\subsection{Growth estimate for the auxiliary process}\label{S:GE}

	To estimate the integral term in  inequality \eqref{22}, we use      the following  weighted energy functional 
		\begin{align}\label{13}
		\mathcal{E}^{\psi}_u(t):=\|u(t)\|^2+\|\psi(t)u(t)\|^2+a\wedge\nu_1\int_0^{t}\left(\|u(s)\|^2_{H^1}+\|\psi(s)u(s)\|^2_{H^1}\right)\dd s.
		\end{align}Based on the growth estimates derived in the Appendix for this functional, we introduce the stopping time 
		\begin{equation}\label{13a}
		\tau^u:=\inf\left\{t\ge0\,|\, \mathcal{E}^{\psi}_u(t)\ge(K+L)t+\rho+M\|u(0)\|^2\right\},
		\end{equation}
		 where $K$ and $M$ are large enough constants so that the results of Proposition~\ref{propositiona1} and  Corollaries~\ref{corollarya1} and \ref{corollaryb1} hold,   
		 while~$\rho>0$ and $L\ge0$ are parameters which will be determined in Section~\ref{proofoftheorem2}. In this subsection, we establish an estimate for the stopping time~$\tau^v$, where~$\{v(t)\}_{t\ge0}$ is the solution of~\eqref{8}.
	\begin{proposition}\label{proposition2}  
	There exist  constants $\gamma, C>0$ and an integer $N\ge1$
	 depending on the parameters $a,\nu,\alpha,q,L,h,\bb_1,\bb_2$ such that the following inequality holds 
	 		\begin{align}\label{33}
		\mathbb{P}\left\{\tau^{v}<\infty\right\}\le 9e^{-\gamma \rho}+\frac{1}{2}\left(\exp\left(Ce^{C\left(\rho+d^2\right)}d^2\right)-1\right)^{\frac{1}{2}}\end{align}
    for any $\rho>0$ and $u,u'\in H$ with $d:=\|u-u'\|$, provided that  \eqref{E:bN} is satisfied. 
	\end{proposition}
	
	\begin{proof} {\it Step 1.} We first reduce the proof to an estimate for a truncated version of~$v$ (see Section~4.2 in~\cite{KN2013} for a similar argument). Let $\{u(t)\}_{t\ge 0}$ and~$\{u'(t)\}_{t\ge0}$ be the solutions of \eqref{1} issued from $u$ and $u'$, respectively, and let us introduce the truncated processes $\{\hat{u}(t)\}_{t\ge0}$, $\{\hat{u}'(t)\}_{t\ge0}$, and~$\{\hat{v}(t)\}_{t\ge0}$ defined as follows: for $t\le \tau$, where $\tau:=\tau^v\wedge\tau^{u}\wedge\tau^{u'}$, they coincide  with $\{u(t)\}_{t\ge0}$, $\{u'(t)\}_{t\ge0}$, and~$\{v(t)\}_{t\ge0}$, respectively, while for $t\ge \tau$, they solve the equation
		\begin{align}\label{34}\partial_t z+az=\nu\partial_{xx} z. \end{align}
		Notice that, by the above definitions, 
		\[\{\tau^v<\infty\}\mcap\{\tau^u=\infty\}\mcap\{\tau^{u'}=\infty\}\subset \{\tau^{\hat{v}}<\infty\}.\]
		Therefore, 
		\begin{align}\label{34a}\{\tau^{v}<\infty\}\subset \{\tau^{\hat{v}}<\infty\}\mcup\{\tau^{u}<\infty\}\mcup\{\tau^{u'}<\infty\}.
		\end{align} Let us take $\gamma:=\gamma_4$ and recall that $\gamma_4\le \gamma_3$, where $\gamma_3$ and $\gamma_4$ are the numbers in Corollaries~\ref{corollarya1} and~\ref{corollaryb1}.  According to Corollary~\ref{corollarya1} with $l=0$ and \eqref{34a}, 
		\begin{align}\label{35}
		\mathbb{P}\{\tau^{v}<\infty\}\le \mathbb{P}\{\tau^{\hat{v}}<\infty\}+6e^{-\gamma \rho}\end{align}
		for any $\rho>0$.  Thus we need to bound the term $\mathbb{P}\{\tau^{\hat{v}}<\infty\}$. 
		
		\noindent{\it Step 2.} In this step, we reduce the bound for $\mathbb{P}\{\tau^{\hat{v}}<\infty\}$ to an estimate for some  measurable transform. Without loss of generality, we assume that the underlying probability space $(\Omega,\mathscr{F},\mathbb{P})$ is of a particular form: $\Omega:= C_0([0,\infty);H)$ is the space of all continuous functions  taking values in $H$ and vanishing at $t=0$, $\mathbb{P}$ is the distribution of the Wiener process
		\begin{equation}\label{EE:xi}
					\xi(t):=\sum_{j=1}^{\infty}b_j\beta_j(t) e_j,
		\end{equation}
		$\mathscr{F}$ is the completion of the Borel $\sigma$-algebra of $\Omega$ associated with the topology of uniform convergence on every compact set. Let us take any integer $N\ge1$ (to~be specified in Step 4) and define the transform 
		\begin{align}\label{36}
            \Phi^{u,u'}: &\ \Omega\to\Omega\notag\\
            &\omega_t\mapsto\omega_t-\int_0^t\I_{\{s\le \tau\}}\ppP_N[\alpha(|\hat{u}|^{q}\hat{u}-|\hat{v}|^{q}\hat{v})-\nu\partial_{xx}(\hat{u}-\hat{v})]\dd s,\end{align}  
		where we use the superscript $u,u'$ to emphasize the dependence of $\Phi^{u,u'}$ on the initial points $u,u'$. Due to the pathwise uniqueness for the stochastic CGL equation, we have
		\begin{align}\label{37}
		\mathbb{P}\left\{\hat{u}'(\Phi^{u,u'}(\omega),t)=\hat{v}(\omega,t),\,\forall t\ge0\right\}=1.\end{align}
		Therefore, 
		\begin{align}\label{37a}\mathbb{P}\{\tau^{\hat{v}}<\infty\}=\Phi^{u,u'}_{*}\mathbb{P}\{\tau^{\hat{u}'}<\infty\}\le \mathbb{P}\{\tau^{\hat{u}'}<\infty\}+\|\mathbb{P}-\Phi^{u,u'}_*\mathbb{P}\|_{\textup {var}},\end{align}
		where $\Phi^{u,u'}_*\mathbb{P}$ denotes the push-forward measure of $\mathbb{P}$ under the transform $\Phi^{u,u'}$.   An estimate for   $\pP\{\tau^{\hat{u}'}<\infty\}$ is provided in Corollary \ref{corollaryb1}. Hence, it remains to bound the total variation distance between $\mathbb{P}$ and~$\Phi^{u,u'}_*\mathbb{P}$.
		
		\noindent {\it Step 3.} To bound $\|\mathbb{P}-\Phi^{u,u'}_*\mathbb{P}\|_{\textup {var}}$, we adopt the strategy presented in Section~3.3.3 of \cite{KS2012} which is   based on the Girsanov theorem. For completeness, let us briefly recall the main ingredients. Notice that $\Omega$ can be expressed as the following direct sum:
		\[\Omega=C([0,\infty);\ppP_NH)\oplus C([0,\infty);\qqQ_NH),\]
		where $\ppP_NH$ and  $\qqQ_NH$ denote  the images of $\ppP_N$ and $\qqQ_N$, respectively. For~$\omega\in\Omega$, we shall accordingly write $\omega=(\omega^{(1)},\omega^{(2)})$. Then, the transformation $\Phi^{u,u'}$ defined in \eqref{36} can be written as 
		\[\Phi^{u,u'}(\omega^{(1)},\omega^{(2)})=(\Psi^{u,u'}(\omega^{(1)},\omega^{(2)}),\omega^{(2)}),\]
		where $\Psi^{u,u'}: \Omega\to C([0,\infty);\ppP_NH)$ is given by 
		\[\Psi^{u,u'}(\omega^{(1)},\omega^{(2)})_t:=\omega_t^{(1)}+\int_0^t\aA(s;\omega^{(1)},\omega^{(2)})\dd s\]
		with 
		\begin{align}\label{39}\aA(t):=-\I_{\{t\le\tau\}}\ppP_N[\alpha(|\hat{u}|^{q}\hat{u}-|\hat{v}|^{q}\hat{v})-\nu\partial_{xx}(\hat{u}-\hat{v})].\end{align}
		Define $\mathbb{P}_N:=(\ppP_N)_*\mathbb{P}$ and $\mathbb{P}_N^{\perp}:=(\qqQ_N)_*\mathbb{P}$, where $\ppP_N$ and $\qqQ_N$ are viewed as the natural projections:
		\[ \ppP_N:\Omega:\to C_0([0,\infty);\ppP_NH),\qquad \qqQ_N:\Omega:\to C_0([0,\infty);\qqQ_NH).\]
		Applying Lemma~3.3.13 in~\cite{KS2012} (more precisely, its version for the current setting), we have 
		\begin{align}\label{39a}\|\mathbb{P}-\Phi^{u,u'}_*\mathbb{P}\|_{\textup {var}}\le \int_{C_0([0,\infty);\qqQ_NH)}\|\Psi^{u,u'}_*(\mathbb{P}_N,\omega^{(2)})-\mathbb{P}_N\|_{\textup {var}}\mathbb{P}^{\perp}_N(\dd\omega^{(2)}),\end{align}
		where $\Psi^{u,u'}_*(\mathbb{P}_N,\omega^{(2)})$ stands for the pushforward measure of $\mathbb{P}_N$ under the map~$\Psi^{u,u'}(\cdot,\omega^{(2)})$ for each fixed $\omega^{(2)}$. Moreover, in view of the Girsanov theorem (e.g., see Theorem~A.10.1 in~\cite{KS2012}), for each $\omega^{(2)}$, 
		\begin{align}\label{39b}
		\|&\Psi^{u,u'}_*(\mathbb{P}_N,\omega^{(2)})-\mathbb{P}_N\|_{\textup {var}} \nonumber\\
		& \quad\quad\le \frac{1}{2}\left(\left(\mathbb{E}_N\exp\left(6\sup_{1\le j\le N}b_j^{-2}\int_0^{\infty}\|\aA(t;\cdot,\omega^{(2)})\|^2\dd t\right)\right)^{\frac{1}{2}}-1\right)^{\frac{1}{2}},\end{align}
		provided that the following Novikov condition is satisfied, i.e., 
		\begin{align}\label{40}
			\mathbb{E}_N\exp\left(c\int_0^{\infty}\|\aA(t;\cdot,\omega^{(2)})\|^2\dd t\right)<\infty
		\end{align}
		for arbitrary $ c>0$ and $\omega^{(2)}$. Here, $\mathbb{E}_N$ denotes the expectation with respect to~$\mathbb{P}_N$.

		\noindent{\it Step 4.} Before turning to the  verification of Novikov's condition, let us derive a pathwise estimate for $\|w(t)\|^2$ before the stopping time $\tau$, where 
		$$
		w(t):=\hat{u}(t)-\hat{v}(t)=u(t)-v(t), \quad t\in [0,\tau).
		$$  
		Let $C_1$ be the constant in Proposition~\ref{proposition1}. Applying Proposition~\ref{proposition1} with 
	\begin{equation}\label{E:eps}
				\varepsilon:=\frac{a(a\wedge\nu_1)}{4C_1(K+L)},
	\end{equation}
		we find $T>0$ and $N\ge1$ such that \eqref{22} holds. We distinguish two cases.
		
		\noindent{\it Case 1: $\tau\le T$.}  Applying \eqref{23}, \eqref{24},   the nonlinear estimate 
		\begin{align*}\langle \qqQ_Nw, \alpha|v|^{q}v-\alpha|u|^{q}u\rangle &\lesssim_{\alpha, q}\|w\|^2\left(\|u\|_{H^1}^{q}+\|v\|_{H^1}^{q}\right)\\&\le C_{a,\alpha,q}\|w\|^2\left(\|u\|_{H^1}^{2}+\|v\|_{H^1}^{2}\right)+\frac{a}{2}\|w\|^2,\end{align*}  Gronwall's inequality, and \eqref{E:PN},
		we arrive at 
		\begin{align}\label{41}
			\|w(t)\|^2\le &\left(\|u-u'\|^2+\frac{|\nu|^2}{a\nu_1}\|\partial_x\ppP_N(u-u')\|^2\right) \nonumber\\ &\times\exp\left(-at+C_{a,\alpha,q}\int_0^t\left(\|u\|^2_{H^1}+\|v\|^2_{H^1}\right)\dd s\right).
		\end{align}
		  Notice that by the definition of $\tau$, we have 
		\begin{align}\label{42}\mathcal{E}^{\psi}_{\hat{u}}(t)&< (K+L)t+\rho+M\|u\|^2\le (K+L)t+\rho+Md^2,\\  
		\label{43}\mathcal{E}^{\psi}_{\hat{v}}(t)&< (K+L)t+\rho+M\|u'\|^2\le (K+L)t+\rho+Md^2
		\end{align}for $t\in [0,\tau)$.
		 Therefore, 
		\begin{align}\label{44}
			\|w(t)\|^2\le& \left(\|u-u'\|^2+\frac{|\nu|^2}{a\nu_1}\|\partial_x\ppP_N(u-u')\|^2\right) \nonumber\\ &\times\exp\left(-at+C_{a,\nu_1,\alpha,q}[(K+L)T+\rho+Md^2]\right).
		\end{align}
		  Since 
		\begin{align*}
			\|\partial_x \ppP_N(u-u')\|^2=\left\|\sum_{j=1}^{N}\langle u-u',e_j\rangle\partial_x e_j\right\|^2\lesssim_N\|u-u'\|^2,  
		\end{align*}
		 we derive from \eqref{44} that
		\begin{equation}\label{45}\|w(t)\|^2\le C\exp\left(-at+C(\rho+d^2)\right)d^2
		\end{equation}for  any  $t\in[0,\tau)$,
		where $C>0$ is a constant depending on $a,\nu,\alpha,q,h,L,\bb_1,\bb_2$. 
		
		\noindent{\it Case 2:  $\tau> T$.} Inequality   \eqref{45} clearly holds for $t\in[0,T]$. To obtain an upper bound for $\|w\|^2$ on the interval $[T,\tau)$, we apply Proposition~\ref{proposition1} (with $\e$ as in~\eqref{E:eps}) combined with~\eqref{E:PN},~\eqref{42}, and \eqref{43}:
		\begin{align*}
			\|w(t)\|^2\le& \left(\|w(T)\|^2+\frac{|\nu|^2}{a\nu_1}\|\partial_x\ppP_N w(T)\|^2\right)\nonumber\\ &\times\exp\left(-\frac{a}{2}t+\frac a2T+\frac{a(\rho+Md^2)}{2(K+L)}\right)
		\end{align*} for  any  $t\in[T,\tau)$.
		This leads to 
		\begin{equation}\label{45z}
		\|w(t)\|^2\le C\exp\left(-\frac a2t+C(\rho+d^2)\right)d^2
		\end{equation}for  any  $t\in[0,\tau)$.
		
		\noindent{\it Step 5.} Now we can verify the Novikov condition \eqref{40}. To this end, let us bound the terms on the right-hand side of \eqref{39}. First,
        \begin{align}\label{45a}
        \|\ppP_N\partial_{xx}w\|^2=\sum_{j=1}^{N}\langle w,\partial_{xx}e_{j}\rangle^2\lesssim_N\|w\|^2.
        \end{align}
        As for the term $\ppP_N(|{v}|^{q}{v}-|{u}|^{q}{u})$, let us consider two cases. If $q\in (1,2)$, then
        \begin{align}\label{45aa}
        \|\ppP_N(|{v}|^{q}{v}-|{u}|^{q}{u})\|^2&=\sum_{j=1}^{N}\langle |{v}|^{q}{v}-|{u}|^{q}{u}),e_{j}\rangle^2\notag\\&\lesssim_q\sum_{j=1}^{N}\langle {|w|}(|{v}|^{q}+|{u}|^{q}),|e_j|\rangle^2\notag\\&\lesssim_{N,q}\langle {|w|},|{v}|^{q}\rangle^2+\langle {|w|},|{u}|^{q}\rangle^2,
        \end{align}where we used that $\{e_j\}$ are bounded functions. The Cauchy--Schwarz,  interpolation, and  Young inequalities imply that 
        \begin{align}\label{45ab}
            \langle {|w|},|{u}|^{q}\rangle^2&\le \|{w}\|^2\|{u}\|_{L^{2q}}^{2q}\lesssim \|{w}\|^2\|{u}\|_{L^{2}}^{2}\|{u}\|^{2q-2}_{H^1}\nonumber\\&\lesssim_q \|{w}\|^2\|{u}\|^2_{H^1}+\|{w}\|^2\|{u}\|^{\frac{2}{2-q}}.
        \end{align}
        Combining \eqref{45aa} and \eqref{45ab}, we arrive at
        \begin{align}\label{45b}
        \|\ppP_N(|{v}|^{q}{v}-|{u}|^{q}{u}))\|^2\lesssim_{N,q}\|{w}\|^2\left(\|{u}\|^2_{H^1}+\|{v}\|^2_{H^1}+\|{u}\|^{\frac{2}{2-q}}+\|{v}\|^{\frac{2}{2-q}}\right).
        \end{align}
        If $q\in (0,1]$, we apply the nonlinear estimate \eqref{27} and Young's inequality to~infer
        \begin{align}\label{45c}
        \|\ppP_N(|{v}|^{q}{v}-|{u}|^{q}{u})\|^2&\le \||{v}|^{q}{v}-|{u}|^{q}{u})\|^2\notag\\&\lesssim_q\|{w}\|^2 \left(\|{v}\|^{2q}_{H^1}+\|{u}\|^{2q}_{H^1}\right)\notag\\&\lesssim_q\|{w}\|^2 +\|{w}\|^2 \left(\|{v}\|^{2}_{H^1}+\|{u}\|^{2}_{H^1}\right).
        \end{align}
        Now let us combine \eqref{45a}, \eqref{45b}, and \eqref{45c}:
        \begin{align}\label{45d}
        \|\aA(t)\|^2&\le \I_{\{t\le\tau\}}\|\ppP_N[\alpha(|{u}|^{q}{u}-|{v}|^{q}{v})-\nu\partial_{xx}({u}-{v})]\|^2\notag\\&\lesssim_{N,\nu,\alpha,q} \I_{\{t\le\tau\}}[\|{w}\|^2+\|{w}\|^2(\|{u}\|^2_{H^1}+\|{v}\|^2_{H^1}+\|{u}\|^{\frac{2}{2-q}}+\|{v}\|^{\frac{2}{2-q}})].
        \end{align}
        To bound the terms on the right-hand side of this inequality, we apply \eqref{45z}: 
        \begin{align}\label{45e}
        \int_0^{\tau}\|{w}(t)\|^2\dd t\le Ce^{C(\rho+d^2)}d^2.
        \end{align}
        Similarly, from \eqref{45z}, \eqref{42}, and \eqref{43}, we obtain
        \begin{align}\label{45f}
        \int_0^{\tau}\|{w}(t)\|^2\left(\|{u}\|^{\frac{2}{2-q}}+\|{v}\|^{\frac{2}{2-q}}\right)\dd t\le Ce^{C\left(\rho+d^2 \right)}d^2.
        \end{align}
        Moreover,  integrating by parts,  
                \begin{align}\label{45g}
         \int_0^{\tau}&\|{w}(t)\|^2\left(\|{u}\|_{H^1}^{2}+\|{v}\|^{2}_{H^1}\right)\dd t\notag\\&\le Ce^{C(\rho+d^2)}d^2\int_0^{\tau}e^{-{\frac a2t}}\dd\left(\int_0^t\left(\|{u}(s)\|_{H^1}^{2}+\|{v}(s)\|^{2}_{H^1}\right)\dd s\right)\notag\\&\le Ce^{C(\rho+d^2)}d^2\left(e^{-{\frac a2\tau}}\left(\mathcal{E}_{{u}}^{\psi}(\tau)+\mathcal{E}_{{v}}^{\psi}(\tau)\right)-{\frac a2}\int_0^{\tau }e^{-{\frac a2t}}\left(\mathcal{E}_{{u}}^{\psi}(t)+\mathcal{E}_{{v}}^{\psi}(t)\right)\dd t\right)\notag\\&\le Ce^{C(\rho+d^2)}d^2.
        \end{align}
        Finally, applying \eqref{45d}-\eqref{45g}, we get
        \begin{align}\label{45h}\mathbb{E}_N\exp\left(c\int_0^{\infty}\|\aA(t;\cdot,\omega^{(2)})\|^2\dd t\right)\le  \exp\left(cCe^{C\left(\rho+d^2\right)}d^2\right),\end{align}
        which is the desired Novikov condition. Besides,   combining \eqref{39a}, \eqref{39b}, and~\eqref{45h}, we derive an estimate for the term $\|\mathbb{P}-\Phi^{u,u'}_*\|_{\textup {var}}$:
$$			\|\mathbb{P}-\Phi^{u,u'}_*\mathbb{P}\|_{\textup {var}}\le \frac{1}{2}\left(\exp\left(Ce^{C\left(\rho+d^2\right)}d^2\right)-1\right)^{\frac{1}{2}}.
$$
	This, together with	  \eqref{35}, \eqref{37a},  and  Corollary~\ref{corollaryb1} 
	 implies \eqref{33} and completes the proof.
	\end{proof}

\section{Proof of Theorem \ref{theorem2}}\label{proofoftheorem2}
	In this section, we establish Theorem \ref{theorem2}. As it was already mentioned, this will imply Theorem \ref{theorem1} and thus Main Theorem.

	\subsection{Recurrence}\label{recurrence}
	
	Before proceeding with the verification of the recurrence property, let us show that the Markov family $(u(t),\mathbb{P}_u)$ corresponding to~\eqref{1} is irreducible. Recall that $P_t(u,\Gamma):=\mathbb{P}\{S_t(u,\cdot)\in \Gamma\}$ is the transition function of   $(u(t),\mathbb{P}_u)$.
	\begin{lemma}\label{lemma2} For any $R,d>0$, there exist constants $p,T>0$ and an integer $N\ge1$ depending on $R,d,a,\nu,\alpha,q,h,\bb_1,\bb_2$ such that 
		\begin{align}\label{46c}
        P_{T}(u_0,B_{H}(0,d))\ge p
        \end{align}
		for all $u_0\in B_H(0,R)$, provided that \eqref{E:bN} holds.
	\end{lemma}
	\begin{proof} {\it Step 1.} Let us take  $T,\delta>0$ and denote $y_1(t):=th+\xi(t)$ and
		\[\Gamma_{\delta}:=\left\{\sup_{t\in[0,T]}\|y_1(t)\|_{H^1}\le \delta\right\},\] where $\xi$ is the Wiener process in \eqref{EE:xi}. Let $y_2(t):=u(t)-y_1(t)$. In this step, we show that for any $T>0$, there exists sufficiently small $\delta>0$ such that~$\|y_2(t)\|$ is uniformly bounded on $[0,T]\times \Gamma_{\delta}$. Indeed, note that~$y_2$ satisfies the equation 
	\[\begin{cases}
		\partial_t y_2 +a(y_1+y_2)-\nu\partial_{xx} (y_1+y_2)+\alpha|u|^qu=0,\\
		y_2(0)=u_0.
	\end{cases}\] Taking the scalar product in $H$ of this equation  with $y_2$ and integrating by parts,   we derive 
    \begin{align}\label{46d}
    \frac{1}{2}\frac{\dd}{\dd t}\|y_2\|^2+a\|y_2\|^2+\nu_1\|y_2\|^2=I_1+I_2,
    \end{align}
    where 
    \begin{align}\label{46e}
    I_1:=\langle -ay_1+\nu\partial_{xx} y_1,y_2\rangle\le  C_{a,\nu}\|y_1\|^2_{H^1}+\frac{a}{2}\|y_2\|^2+\frac{\nu_1}{2}\|\partial_x y_2\|^2
    \end{align}
    and
    \begin{align}\label{46f}
    I_2&:=\langle -\alpha |u|^q u, y_2\rangle =-\langle \alpha |u|^q y_1,y_2\rangle-\alpha _1\langle |u|^q ,|y_2|^2\rangle\nonumber\\&\le C_{\alpha,q
    }\langle |y_1|^q+|y_2|^q,|y_1||y_2|\rangle. 
    \end{align}
    We claim that,
	 for any $v\in H$, 
	\begin{align}\label{46g}
		\langle|v|^q,|y_1||y_2|\rangle\le C_{a,\nu,\alpha,q}\delta^2\left(\|v\|^{2q}+\|v\|^{ 2}\right)+C_{a,\nu,\alpha,q}\delta^2+\frac{a\wedge\nu_1}{4C_{\alpha,q}}\|y_2\|^2_{H^1}
	\end{align}
	on $\Gamma_{\delta}$. Indeed, if $q\in(0,1)$, applying the Young and H\"older inequalities, we get
	\begin{align*}
	  	\langle|v|^q,|y_1||y_2|\rangle&=\langle|v|^q|y_1|^q,|y_2||y_1|^{1-q}\rangle \le C_{q} \|v\|^2\|y_1\|_{L^{\infty}}^2+C_{q} \langle |y_2|^{\frac{2}{2-q}},|y_1|^{\frac{2(1-q)}{2-q}}\rangle\notag\\&\le C_{a,\nu, \alpha,q} \delta^2\|v\|^2+C_{a,\nu, \alpha,q}\delta^2 +\frac{a\wedge\nu_1}{4C_{\alpha,q}}\|y_2\|^2.
	\end{align*}
	On the other hand, if $q\in [1,2)$,   
	\begin{align*}
		\langle|v|^q,|y_1||y_2|\rangle&\le  C \delta^{q-1}\|y_2\|_{H^1} \langle|v|^q,|y_1|^{2-q}\rangle\le C\delta^{q-1}\|y_2\|_{H^1} \|v\|^q\|y_1\|^{2-q}\notag\\&\le C_{a,\nu,\alpha,q} \delta^2\|v\|^{2q}+\frac{a\wedge\nu_1}{4C_{\alpha,q}}\|y_2\|^2_{H^1}.
	\end{align*}
    Combining \eqref{46d}-\eqref{46g}, we obtain 
    \[\frac{\dd}{\dd t}\|y_2\|^2 \le C_{a,\nu,\alpha,q}\delta^2\left(\|y_2\|^{2}+\|y_2\|^{2q}+1\right)\le C_{a,\nu,\alpha,q}\delta^2\left(\|y_2\|^2+1\right)^2
    \]for $\delta\le1$. Applying the nonlinear Gronwall inequality and choosing $\delta$ sufficiently small, we arrive at
\begin{equation}\label{46h}
    \sup_{t\in[0,T]}\|y_2(t)\|^2\le  2(1+R^2) \quad \textup{on $\Gamma_{\delta}$}.	
\end{equation}

    \noindent{\it Step 2.} Now we prove that there exist sufficiently large $T>0$ and small $\delta>0$ such that $\|u(T)\|<d$ on the event~$\Gamma_{\delta}$. Let $y_3$ be the solution of the unforced CGL equation 
	\[\begin{cases}
		\partial_t y_3 +ay_3-\nu\partial_{xx} y_3+\alpha|y_3|^qy_3=0,\\
		y_3(0)=u_0.
	\end{cases}\] 
	Then, the standard energy estimate gives 
\begin{equation}\label{E:ERE}
		\|y_3(t)\|\le e^{-at }\|u_0\|.
\end{equation}
	By choosing 
	\[T:=\frac{1}{a}\log \frac{4R}{d},\]
	we have 
	\[\|y_3(T)\|\le \frac{d}{4}.\]
	Now we estimate the difference $y(t):=y_2(t)-y_3(t)$. The latter is the solution of the equation 
	\[\begin{cases}
		\partial_t y+a(y_1+y)-\nu\partial_{xx} (y_1+y)+\alpha(|u|^qu-|y_3|^q y_3)=0,\\
		y(0)=0.
	\end{cases}\]Taking the scalar product in $H$ of this equation  with $y$ and integrating by parts, we get
	\begin{align}\label{47}
		\frac{1}{2}\frac{\dd}{\dd t}\|y\|^2+a\|y\|^2+\nu_1\|\partial_x y\|^2=J_1+J_2,
	\end{align}
	where 
	\begin{align}\label{48}J_1:= \langle \nu\partial_{xx} y_1-ay_1, y\rangle\le \frac{a}{4}\|y\|^2+\frac{\nu_1}{4}\|\partial_x y\|^2+C_{a,\nu}\|y_1\|^2_{H^1}\end{align}
	and 
	\begin{align}\label{49}
	J_2&:=\langle  \alpha |y_3|^q y_3-\alpha|u|^qu,y\rangle\nonumber \\&\le C_{\alpha,q}\langle |y_1|^q+|y_3|^q+|y|^q,|y_1|^2+|y|^2\rangle.
		\end{align}
	Notice that 
    \[\langle |v_1|^q, |v_2|^2\rangle\le C\|v_1\|^q\|v_2\|_{H^1}^q\|v_2\|^{2-q}\]for $v_1\in H$ and  $v_2\in H^1$. Therefore,
	\begin{align}
		\langle |y_1|^q+|y_3|^q+|y|^q,|y_1|^2\rangle &\le C_{a,\alpha,q,R}\delta^2+\frac{a}{4C_{\alpha,q}}\|y\|^2,\label{50a}\\
	 		\langle |y_1|^q+|y_3|^q,|y|^2\rangle &\le C_{q,R}\|y\|_{H^1}^q\|y\|^{2-q}\nonumber\\&\le C_{a,\nu,\alpha,q,R}\|y\|^2+\frac{a\wedge\nu_1}{4C_{\alpha,q}}\|y\|^2_{H^1},\label{50b}
	\\\label{50c}
		\langle |y|^q,|y|^2\rangle\le C\|y\|_{H^1}^q\|y\|^2&\le C_{a,\nu,\alpha,q}\|y\|^{\frac{4}{2-q}}+\frac{a\wedge\nu_1}{4C_{\alpha,q}}\|y\|^2_{H^1}.
	\end{align}
	Combining \eqref{E:ERE}-\eqref{50c} and choosing $\delta$ so small that  \eqref{46h} holds, we infer that
	\begin{align*}
		\frac{\dd}{\dd t}\|y(t)\|^2 &\le C_{a,\nu,\alpha,q,R}\left(\delta^2+\|y\|^{2}+\|y\|^{\frac{4}{2-q}}\right)\\&\le C_{a,\nu,\alpha,q,R}\left(\delta^2+\|y\|^{2}+\left(\|y_2\|+\|y_3\|\right)^{\frac{2q}{2-q}}\|y\|^2\right)\\&\le C_{a,\nu,\alpha,q,R}\left(\delta^2+\|y\|^2\right)
	\end{align*}
    on $\Gamma_{\delta}$, which implies  
	\[\|y(T)\|^2 \le \delta^2\exp(C_{a,\nu,\alpha,q,R}T).\]
	Choosing $\delta$ so small that also  
	\begin{align*}\delta<\frac{d}{2},\qquad \delta^2\exp(C_{a,\nu,\alpha,q,R}T)\le \frac{d^2}{16},\end{align*}
	we obtain  
		 $\|u(T)\|<d$ on $\Gamma_{\delta}$. 
	
	\noindent{\it Step 3.} In the previous step, we have shown that  
		\[P_T(u_0,B_H(0,d))\ge \mathbb{P}(\Gamma_{\delta})\]
	for all $u_0\in B_H(0,R)$. Therefore, it remains to prove that $\mathbb{P}(\Gamma_{\delta})>0$. Let us choose $N\ge 1$ large enough so that 
	\[\|\qqQ_{N} h\|_{H^1}\le \sum_{j\ge N+1} |\langle h,e_j\rangle|\|e_j\|_{H^1}<\frac{\delta}{3T}; \]
	this choice is possible in view of \eqref{E:0.3}.
	Then, by independence of $\ppP_N\xi$ and $\qqQ_N\xi$, 
	$$\mathbb{P}(\Gamma_{\delta})\ge \mathbb{P}\left\{\sup_{t\in[0,T]}\|\ppP_{N} ht+\ppP_N\xi(t)\|_{H^1}<\frac{\delta}{3}\right\}\times \mathbb{P}\left\{\sup_{t\in[0,T]}\|\qqQ_N\xi(t)\|_{H^1}<\frac{\delta}{3}\right\}.
	$$
	The first term on the right-hand side of this inequality is positive, provided that~$b_j\neq 0$ for all $j\le N$. As for the second term, the assumptions $\bb_1,\bb_3<\infty$ ensure that the distribution of $\qqQ_N\xi$ is a centered Gaussian measure on the space~$C([0,T];H^1)$.  Therefore, we have  
	\[\mathbb{P}\left\{\sup_{t\in[0,T]}\|\qqQ_{N}\xi(t)\|_{H^1}<\frac{\delta}{3}\right\}>0,\]
	which implies the positivity of $\mathbb{P}(\Gamma_{\delta})$ as desired. This completes the proof.
	\end{proof}
    Next, we show that the extension $(\bm{u}(t),\mathbb{P}_{\bm{u}})$ of $(u(t),\mathbb{P}_{u})$ constructed in Section~\ref{maintheorem} is also irreducible. 
    \begin{proposition} \label{proposition3}
   For any $R,d>0$, there exist constants $p,T>0$ and an integer $N\ge1$
   depending on $R,d,a,\nu,\alpha,q,h,\bb_1,\bb_2$
    such that
		\begin{align}\label{51a}
		\mathbb{P}_{\bm{u}}\left\{\bm{u}(T)\in B_{H}(0,d)\times B_{H}(0,d)\right\}\ge p
		\end{align}
		for all $\bm{u}\in B_{H}(0,R)\times B_{H}(0,R)$,   provided that \eqref{E:bN} holds.
 \end{proposition}
		\begin{proof}
    \noindent{\it Step 1.} First, let us notice that \eqref{51a} holds trivially when $R<d$. Let~$\tilde{u},\tilde{u}',\tilde{v}$ be the processes defined in Section \ref{maintheorem} and let   $R\ge d>0$ and~$\rho>0$ be any numbers. Let us    introduce the following events:      
    \[G_d(T):=\{\|\tilde{u}(T)\|\le d\},\qquad G'_d(T):=\{\|\tilde{u}'(T)\|\le d\},\]
    \[E_{\rho}:=\{\mathcal{E}^{\psi}_{\tilde{u}}(t)< Kt+\rho+C\|u\|^2,\forall t\ge 0 \}\mcap \{\mathcal{E}^{\psi}_{\tilde{u}'}(t)< Kt+\rho+C\|u'\|^2,\forall t\ge 0 \},\]  
    where $K,C$ are the constants in Proposition~\ref{propositiona1}, and $\mathcal{E}^{\psi}_u$ is defined in \eqref{13}. Then, we~have 
    \[\mathbb{P}_{\bm{u}}\left\{\bm{u}(T)\in B_{H}(0,d)\times B_{H}(0,d)\right\}=\mathbb{P}_{\bm{u}}\left(G_d(T)G'_d(T)\right).\]
    To simplify the notation, here and in what follows, we write 
    \[AB:=A\mcap B\]
    for any events $A,B$. Moreover, since $(\bm{u}(t),\mathbb{P}_{\bm{u}})$ is an extension of~$(u(t),\mathbb{P}_u)$, Lemma~\ref{lemma2} gives a lower bound for $\mathbb{P}_{\bm{u}}(G_d(T))$. However, it is not clear whether a similar lower bound holds for $\mathbb{P}_{\bm{u}}\left(G_d(T) G'_d(T)\right)$, because the events $G_d(T)$ and~$G'_d(T)$ are not necessarily independent. To remedy this difficulty, we need to use the structure of the extension $(\bm{u}(t),\mathbb{P}_{\bm{u}})$ to perform a detailed analysis. Before proceeding, let us fix $N\ge1$ large enough so that Lemma~\ref{lemma2} and Proposition~\ref{proposition1} hold with $\varepsilon=\frac{a(a\wedge\nu_1)}{4C_1K}$, where $C_1$ is the constant in Proposition~\ref{proposition1}. Then, by applying Lemma~\ref{lemma2}, we obtain constants $p_0,T_0>0$ such that
    \begin{align}\label{51b}
    \mathbb{P}_{u}\left\{\|u(T_0)\|\le d/2\right\}\ge p_0\end{align} for
    any $u\in B_{H}(0,R)$. From the strong Markov property it follows that (cf.~Step~1 of the proof of Proposition~5.1 in~\cite{M2014}) 
    \begin{align}\label{51c}\mathbb{P}_{u}\left\{\|u(kT_0)\|\le d/2\right\}\ge p_0
    \end{align}
    for any $u\in B_{H}(0,R)$ and $k\ge1$, which implies
    \begin{align}\label{51d}
    \mathbb{P}_{\bm{u}} (G_{d/2}(kT_0) )\wedge\mathbb{P}_{\bm{u}} (G'_{d/2}(kT_0) )\ge p_0
    \end{align}
    for any  $\bm{u}\in B_{H}(0,R)\times B_{H}(0,R)$ and $k\ge1$.

    \noindent{\it Step 2.}  Let us set   
    \[\mathcal{N}:=\{\tilde{v}(t)\neq \tilde{u}'(t)\ \mbox{for some $t\ge0$}\}.\]
         We claim that for arbitrary   $\rho>0$, there exists $k\ge 1$ such that
    \begin{align}\label{51e}G_{d/2}(kT_0)E_{\rho}\mathcal{N}^\mathsf{c}\subset G_d(kT_0) G_d'(kT_0)
    \end{align} for any $\bm{u}\in B_{H}(0,R)\times B_{H}(0,R)$.
    Indeed, it suffices to show that
    \[\|\tilde{v}(kT_0)-\tilde{u}(kT_0)\|\le d/2\]
    on the event $G_{d/2}(kT_0)E_{\rho}\mathcal{N}^\mathsf{c}$ for large enough $k\ge1$ and any $\bm{u}\in B_{H}(0,R)\times B_{H}(0,R)$. Notice that, since $\tilde{v}$ and $\tilde{u}$ solve the equations \eqref{9} and \eqref{11}, their difference $\tilde{w}=\tilde{u}-\tilde{v}$ solves the equation \eqref{22a} with $u,v$ replaced by $\tilde{u},\tilde{v}$. By applying Proposition \ref{proposition1} and repeating the argument for the derivation of \eqref{45z}, we get
    \begin{align}\label{51f}
    \|\tilde{w}(t)\|^2\le Ce^{-\frac{at}{2}},\qquad t\ge0,\end{align}
	where $C$ is a positive constant depending on $a,\nu,\alpha,q,h,\bb_1,\bb_2,R,\rho$. Plugging $t= kT_0$ in \eqref{51f}, we see that there exists large enough $k\ge1$ such that 
    \[\|\tilde{w}(kT_0)\|\le d/2\]
    on $G_{d/2}(kT_0)  E_{\rho}  \mathcal{N}^\mathsf{c}$ as desired, which implies \eqref{51e}.

   Since \eqref{51e} holds for any $\bm{u}\in B_{H}(0,R)\times B_{H}(0,R)$,  we have also 
    $$G_{d/2}'(kT_0)E_{\rho} \mathcal{N}^\mathsf{c}\subset G_d(kT_0)G_d'(kT_0)
  $$  for the same choice of~$k$.
  
    \noindent{\it Step 3.} Now we are able to prove \eqref{51a}.  Without loss of generality, we can assume that    
    \begin{align}\label{51h}
    \mathbb{P}_{\bm{u}}\left(G'_{d/2}(kT_0) \mathcal{N}^\mathsf{c}\right)\le \mathbb{P}_{\bm{u}}\left(G_{d/2}(kT_0)\mathcal{N}^\mathsf{c}\right).
    \end{align}
    Hence, we derive from \eqref{51e} that
    \begin{align*}
    \mathbb{P}_{\bm{u}}&\left(G_{d}(kT_0)  G'_{d}(kT_0)\right)\\&=\mathbb{P}_{\bm{u}}\left(G_{d}(kT_0)  G'_{d}(kT_0)  \mathcal{N}^\mathsf{c}\right)+\mathbb{P}_{\bm{u}}\left(G_{d}(kT_0)  G'_{d}(kT_0)  \mathcal{N}\right)\\&\ge \mathbb{P}_{\bm{u}}\left(G_{d}(kT_0)  G'_{d}(kT_0)  \mathcal{N}^\mathsf{c}  E_{\rho}\right)+\mathbb{P}_{\bm{u}}\left(G_{d}(kT_0)  G'_{d}(kT_0)| \mathcal{N}\right)\mathbb{P}_{\bm{u}}(\mathcal{N})\\&\ge \mathbb{P}_{\bm{u}}\left(G_{d/2}(kT_0)  \mathcal{N}^\mathsf{c}  E_{\rho}\right)+\mathbb{P}_{\bm{u}}\left(G_{d}(kT_0)| \mathcal{N}\right)\mathbb{P}_{\bm{u}}\left(G_{d}'(kT_0)| \mathcal{N}\right)\mathbb{P}_{\bm{u}}(\mathcal{N})\\&\ge \mathbb{P}_{\bm{u}}\left(G_{d/2}(kT_0)  \mathcal{N}^\mathsf{c}  E_{\rho}\right)+\mathbb{P}_{\bm{u}}\left(G_{d}(kT_0)  \mathcal{N}\right)\mathbb{P}_{\bm{u}}\left(G_{d}'(kT_0)  \mathcal{N}\right),
    \end{align*}
    where we used the independence of $\tilde{v}$ and $\tilde{u}'$ conditioned on $\mathcal{N}$. The latter comes from the maximal coupling construction (see Section~1.2.4 in~\cite{KS2012}). Now let us fix 
    \[\rho:=\frac{1}{\gamma}\log\left(\frac{48}{p_0^2}\right),\]
    where $\gamma$ is the constant given in Proposition \ref{propositiona1}. Then, by Proposition~\ref{propositiona1}, we~have 
    \[\mathbb{P}\left(E_{\rho}\right)\ge 1-\frac{p_0^2}{8},\]
     which implies 
    \begin{align}\label{51i} \mathbb{P}_{\bm{u}}&\left(G_{d}(kT_0)  G'_{d}(kT_0)\right)\notag\\&\ge \mathbb{P}_{\bm{u}}\left(G_{d/2}(kT_0)  \mathcal{N}^\mathsf{c}\right)+\mathbb{P}_{\bm{u}}\left(G_{d}(kT_0)  \mathcal{N}\right)\mathbb{P}_{\bm{u}}\left(G_{d}'(kT_0)  \mathcal{N}\right)-\frac{p_0^2}{8}\end{align}
    for large enough $k\ge1$. Further, we claim that the right-hand side of \eqref{51i} is no less than $\frac{p_0^2}{8}$. Indeed, if \[\mathbb{P}_{\bm{u}}\left(G_{d/2}(kT_0)  \mathcal{N}^\mathsf{c}\right)\ge \frac{p_0^2}{4},\]
    then the desired result follows directly. Otherwise, \eqref{51d} and \eqref{51h} imply
    \begin{align*}
    p_0^2&\le \mathbb{P}_{\bm{u}}\left(G_{d/2}(kT_0)\right)\mathbb{P}_{\bm{u}}\left(G_{d/2}'(kT_0)\right)\\&\le \mathbb{P}_{\bm{u}}\left(G_{d/2}(kT_0)  \mathcal{N}\right)\mathbb{P}_{\bm{u}}\left(G_{d/2}'(kT_0)  \mathcal{N}\right)+3\mathbb{P}_{\bm{u}}\left(G_{d/2}(kT_0)  \mathcal{N}^\mathsf{c}\right)\\&\le \mathbb{P}_{\bm{u}}\left(G_{d/2}(kT_0)  \mathcal{N}\right)\mathbb{P}_{\bm{u}}\left(G_{d/2}'(kT_0)  \mathcal{N}\right)+\frac{3p_0^2}{4}.
    \end{align*}
    Therefore, we have 
    \begin{align*}
    \mathbb{P}_{\bm{u}}\left(G_{d}(kT_0)  \mathcal{N}\right)\mathbb{P}_{\bm{u}}\left(G_{d}'(kT_0)  \mathcal{N}\right)\ge \mathbb{P}_{\bm{u}}\left(G_{d/2}(kT_0)  \mathcal{N}\right)\mathbb{P}_{\bm{u}}\left(G_{d/2}'(kT_0)  \mathcal{N}\right)\ge\frac{p_0^2}{4},
    \end{align*}
    which implies 
    \begin{align*}\mathbb{P}_{\bm{u}}\left(G_{d}(kT_0)  G'_{d}(kT_0)\right)\ge \frac{p_0^2}{8}\end{align*}
    as desired. This completes the proof with $T:=kT_0$ and $p:=\frac{p_0^2}{8}$.
		\end{proof}
		We verify the recurrence property \eqref{14} by combining the irreducibility property established in the previous proposition with the existence of a Lyapunov function for the family $(u(t),\mathbb{P}_u)$ (see Section~3 in~\cite{S2008}).	
	  \begin{definition}
	  	Let $(y(t),\mathbb{P}_y)$ be a Markov family in a separable Banach space~$X$, and let~$F:X\to [1,+\ty)$ be a continuous function  such that
	\[\lim_{\|y\|_{X}\to+\infty}F(y)=+\infty.\]
	Then, $F$ is said to be a Lyapunov function for $(y(t),\mathbb{P}_y)$, if there are positive constants $t_*,R_*,C_*,$ and $q_*<1$ such that 
	\begin{align}\label{52}
		\mathbb{E}_y F(y(t_*))\le q_*F(y)\qquad \mbox{for $\|y\|_{X}\ge R_*$},
	\end{align}
	\begin{align}\label{53}
		\mathbb{E}_y F(y(t))\le C_*\qquad \mbox{for $\|y\|_{X}\le R_*, t\ge0$}.
	\end{align}

	  \end{definition}  	\begin{proposition}\label{proposition4}
		The family $(u(t),\mathbb{P}_u)$ corresponding to the CGL equation \eqref{1} possesses a Lyapunov function given by 
		\[F(u)=1+\|u\|^2.\]
	\end{proposition}
    \begin{proof} In view of \eqref{19}, we have
    \begin{align}\label{53a}\mathbb{E}_uF(u(t))\le e^{-at}\|u\|^2+C,\end{align}
    where $C>0$ is a constant depending on $a,h,\bb_1$. Therefore,  
    \[\mathbb{E}_uF(u(t_*))\le e^{-at_*}\|u\|^2+\frac{1}{2}\|u\|^2\le \frac{3}{4}F(u)\]
    for $t_*=\frac{\log 4}{a}$, $q_*=\frac{3}{4}$, and any $u\in H$ with $\|u\|\ge R_*:=2C$. To verify \eqref{53}, we apply \eqref{53a}:   
    \begin{align*}
    \mathbb{E}_uF(u(t))\le e^{-at} R_*^2+C\le  4C^2+C=:C_*.
    \end{align*}
      This completes the proof.
    \end{proof}
	Combining Propositions~\ref{proposition3} and~\ref{proposition4}, we see that the hypotheses of Proposition~3.3 in \cite{S2008} are satisfied for $(\bm{u}(t),\mathbb{P}_{\bm{u}})$. Hence, we conclude that the recurrence property holds. 
    \begin{proposition} For any $d>0$, there exist   constants $\delta,C>0$ and an integer $N\ge1$ depending on $d,a,\nu,\alpha,q,h,\bb_1,\bb_2$ such that 
    \begin{align}
    \mathbb{E}_{\bm{u}}\exp(\delta \tau_d)\le C(1+\|u\|^2+\|u'\|^2)\quad \text{for any $\bm{u}\in H\times H$,}
    \end{align} provided that \eqref{E:bN} holds. 
	      \end{proposition}

	\subsection{Exponential squeezing}\label{squeezing}

    Let $u,u',v,\tilde{u},\tilde{u}',\tilde{v}$ be the processes defined in Section~\ref{maintheorem}, and let $\hat{u},\hat{u}',\hat{v}$ be the truncated processes introduced in the proof of Proposition~\ref{proposition2}.  Let the constants~$K$ and $M$ be as in Subsection~\ref{S:GE}.

  The verification of the exponential squeezing property is inspired by the arguments of~\cite{S2008, M2014}.
  We define the stopping time $\sigma$ in Theorem \ref{theorem2} as follows: 
	\[\sigma:=\tilde{\tau}\wedge\sigma_1,\]
	where $\tilde{\tau}:=\tau^{\tilde{u}}\wedge\tau^{\tilde{u}'}$ and $\sigma_1:=\inf\{t\ge0\,|\, \tilde{u}'(t)\neq \tilde{v}(t)\}$. 
     Let us consider the following events
    \begin{align*}
    	 \mathcal{Q}_k'&:=\{kT\le \sigma\le (k+1)T,\,\sigma_1\ge \tilde{\tau}\},\\\mathcal{Q}_k^{''}&:=\{kT\le \sigma\le (k+1)T,\,\sigma_1<\tilde{\tau}\}, \quad k\ge0.
    \end{align*}
    Before checking the exponential squeezing property for $\sigma$, we show that the probabilities of $\mathcal{Q}_k'$ and $\mathcal{Q}_k''$ decay exponentially in $k$.
    \begin{lemma}\label{lemma3} There exist constants $d,\rho,L, T>0$  and an integer $N\ge1$ depending on $a,\nu,\alpha,q,h,\bb_1,\bb_2$
     such that we have
    \[\mathbb{P}_{\bm{u}}(\mathcal{Q}_k')\vee \mathbb{P}_{\bm{u}}(\mathcal{Q}_k^{''})\le e^{-2(k+1)},\quad k\ge0\]
    for any $\bm{u}\in\overline{B}_{H}(0,d)\times \overline{B}_{H}(0,d)$, 
     provided that \eqref{E:bN} holds.  
    \end{lemma}
    \begin{proof}{\it Step 1: estimate for $\mathbb{P}_{\bm{u}}(\mathcal{Q}_k')$.} In what follows, we fix $\rho$ and $L$ large enough so~that 
    \begin{align}\label{54}\rho\ge \frac{5}{\gamma_4},\qquad L\ge \frac{2}{\gamma_3},\end{align}
    where $\gamma_3,\gamma_4$ are the constants  in Propositions~\ref{propositiona1} and~\ref{propositionb1}, respectively (recall that $\gamma_4\le \gamma_3$).   Applying Corollary~\ref{corollarya1}, we derive 
    \begin{align*}
        \mathbb{P}_{\bm{u}}(\mathcal{Q}_k')\le \mathbb{P}_{\bm{u}}\{kT\le \tilde{\tau}<\infty\} \le \mathbb{P}_u\{kT\le \tau^{\tilde{u}}<\infty\}\le 3e^{-\gamma_3(\rho+LkT)}.
    \end{align*}
Then, from \eqref{54} we obtain that 
    \begin{align}\label{55}\mathbb{P}_{\bm{u}}(\mathcal{Q}'_k)\le 3e^{-2k-5}\le e^{-2(k+1)}
    \end{align}for any $T\ge 1$ and $k\ge0$. 

    \noindent{\it Step 2: estimate for $\mathbb{P}_{\bm{u}}(\mathcal{Q}_0^{''})$.} By definition,
    \begin{align}\label{55aa}\mathbb{P}_{\bm{u}}(\mathcal{Q}_0^{''})&\le \mathbb{P}_{\bm{u}}\{0\le \sigma_1\le T\}=\mathbb{P}_{\bm{u}}\{\tilde{u}'(t)\neq \tilde{v}(t)\ \mbox{for some $t\in[0,T]$}\}\notag\\&= \|\lambda_T(u,u')-\lambda_T'(u,u')\|_{\textup {var}},\end{align}
    where $\lambda_T(u,u')$ and $\lambda'_T(u,u')$ denote the distributions of $\{v(t)\}_{t\in[0,T]}$ and $\{u'(t)\}_{t\in[0,T]}$ respectively, and we use the fact that $\{\tilde{v}(t)\}_{t\in[0,T]}$ and $\{\tilde{u}'(t)\}_{t\in[0,T]}$ are flows of the maximal coupling $(\mathcal{V}_T(u,u'),\mathcal{V}_T'(u,u'))$. Further, we have
    \begin{align}\label{55ab}
        \|&\lambda_T(u,u')-\lambda_T'(u,u')\|_{\textup {var}}=\sup_{\Gamma\in  \mathscr{B}(C([0,T];H))}|\mathbb{P}\{v(\cdot)\in \Gamma\}-\mathbb{P}\{u'(\cdot)\in \Gamma\}|\notag\\&\le \mathbb{P}\{\tau<\infty\}\nonumber\\&\quad+\sup_{\Gamma\in\mathscr{B}(C([0,T];H))}|\mathbb{P}\{v(\cdot)\in \Gamma,\tau=\infty\}-\mathbb{P}\{u'(\cdot)\in \Gamma,\tau=\infty\}|\notag\\&=:\mathcal{L}_1+\mathcal{L}_2,
    \end{align}
    where $\tau:=\tau^{u}\wedge\tau^{u'}\wedge\tau^v$. For $\mathcal{L}_1$, we have 
    \begin{align}\label{55ac}
    \mathcal{L}_1\le \mathbb{P}\{\tau^{u}<\infty\}+\mathbb{P}\{\tau^{u'}<\infty\}+\mathbb{P}\{\tau^{v}<\infty\}.
    \end{align}
    As for $\mathcal{L}_2$, we use the definitions of $\hat{u}'$ and $\hat{v}$ and the equality \eqref{37}:
    \begin{align}\label{55ad}
        \mathcal{L}_2&=\sup_{\Gamma\in\mathscr{B}(C([0,T];H))}|\mathbb{P}\{\hat{v}(\cdot)\in \Gamma,\tau=\infty\}-\mathbb{P}\{\hat{u}'(\cdot)\in \Gamma,\tau=\infty\}|\notag\\&\le  \sup_{\Gamma\in\mathscr{B}(C([0,T];H))}|\mathbb{P}\{\hat{v}(\cdot)\in \Gamma\}-\mathbb{P}\{\hat{u}'(\cdot)\in \Gamma\}|\notag\\&\le \|\mathbb{P}-\Phi_*^{u,u'}\mathbb{P}\|_{\textup {var}},
    \end{align}
    where $\Phi^{u,u'}$ is the transformation defined in \eqref{36}. Hence, in view of \eqref{55aa}-\eqref{55ad} together with Proposition \ref{proposition2} and Corollary \ref{corollarya1}, we derive that
    \begin{align}\label{55ae}
    \mathbb{P}_{\bm{u}}(\mathcal{Q}_0^{''})\le 15e^{-\gamma_4\rho}+\left(\exp\left(Ce^{C(\rho +d^2)}d^2\right)-1\right)^{\frac{1}{2}}.
    \end{align}
      Using the inequality  
    $$
    e^x-1\le 2x\quad\text{ for~$x\in[0,1]$}
    $$ and choosing $d<1$ so small that
    \[Ce^{C(\rho +d^2)}d^2\le e^{-8},\]
    we infer from \eqref{55ae} that
    \[\mathbb{P}_{\bm{u}}(\mathcal{Q}_0^{''})\le 15e^{-5}+\sqrt{2}e^{-4}<  e^{-2}\]
    as desired.
    
    \noindent{\it Step 3: estimate for $\mathbb{P}_{\bm{u}}(\mathcal{Q}_k^{''})$, $k\ge1$.} The estimate for the case when $k\ge1$ is based on the Markov property combined with the arguments used in the estimate for $\mathbb{P}_{\bm{u}}(\mathcal{Q}_0^{''})$. First, let us apply the Markov property to obtain
    \begin{align*} 
    \mathbb{P}_{\bm{u}}(\mathcal{Q}_{k}^{''})&=\mathbb{P}_{\bm{u}}(\mathcal{Q}_{k}^{''},\sigma\ge kT)=\mathbb{E}_{\bm{u}}\left(\I_{\{\sigma\ge kT\}}\mathbb{E}_{\bm{u}}\left(\I_{\mathcal{Q}_{k}^{''}}|\mathscr{F}_{kT}\right)\right)\notag\\&\le \mathbb{E}_{\bm{u}}\left(\I_{\{\sigma\ge kT\}}\mathbb{P}_{\bm{u}(kT)}\{0\le \sigma_1\le T\}\right).
    \end{align*}
    As established in Step 2,  we have 
    \begin{align*} 
    \mathbb{P}_{\bm{u}}\{0\le \sigma_1\le T\}&\le \mathbb{P}\{\tau^{u}<\infty\}+\mathbb{P}\{\tau^{u'}<\infty\}\nonumber\\&\quad+\mathbb{P}\{\tau^{v}<\infty\}+\|\mathbb{P}-\Phi_*^{u,u'}\mathbb{P}\|_{\textup {var}}
    \end{align*}for any   $\bm{u}=(u,u') \in H\times H$.
    Hence, by combining this with   \eqref{34a} and \eqref{37a}, we infer
    \begin{align*}
    \mathbb{P}_{\bm{u}}\{0\le \sigma_1\le T\}&\le 2\mathbb{P}\{\tau^{u}<\infty\}+2\mathbb{P}\{\tau^{u'}<\infty\}\\&\quad +\mathbb{P}\{\tau^{\hat{u}'}<\infty\}+2\|\mathbb{P}-\Phi_*^{u,u'}\mathbb{P}\|_{\textup {var}},
    \end{align*}
    which implies
    \begin{align}\label{56b}
    \mathbb{P}_{\bm{u}}(\mathcal{Q}_k^{''})&\le 2\mathbb{E}_{\bm{u}}\left(\I_{\{\sigma\ge kT\}}\|\mathbb{P}-\Phi_*^{\tilde{u}(kT),\tilde{u}'(kT)}\mathbb{P}\|_{\textup {var}}\right)\notag\\&\quad+2\mathbb{E}_{\bm{u}}\left(\I_{\{\sigma\ge kT\}}\mathbb{P}_{\bm{u}(kT)}\{\tau^{{u}}<\infty\}\right)+2\mathbb{E}_{\bm{u}}\left(\I_{\{\sigma\ge kT\}}\mathbb{P}_{\bm{u}(kT)}\{\tau^{{u}'}<\infty\}\right)\notag\\&\quad+\mathbb{E}_{\bm{u}}\left(\I_{\{\sigma\ge kT\}}\mathbb{P}_{\bm{u}(kT)}\{\tau^{\hat{{u}}'}<\infty\}\right)\notag\\&=:2I_1+2I_2+2I_3+I_4.
    \end{align}
    Let us bound the terms $I_j,$ $j=1,\ldots,4$. For $I_1$, we use arguments similar to those in the proof of Proposition~\ref{proposition2}.     Notice that for any $\omega\in\Omega$, 
    \[\Phi^{\tilde{u}(kT),\tilde{u}'(kT)}(\omega)_t=\omega_t-\int_0^t\I_{\{s\le \tau^{k}\}}\ppP_N[\alpha(|u_k|^{q}u_k-|v_k|^{q}v_k)-\nu\partial_{xx}(u_k-v_k)]\dd s,\]
    where $u_{k},u'_k$ denote the solutions of \eqref{1} issued from $\tilde u(kT), \tilde u'(kT)$, respectively, $v_k$ is the solution of \eqref{8} with initial data given by $\tilde u'(kT)$ and $\tau^k:=\tau^{u_k}\wedge\tau^{u'_k}\wedge\tau^{v_k}$. In view of \eqref{39a} and \eqref{39b}, to bound $I_1$, we need to estimate the integral~$\int_0^{\infty}\|\aA_k(t)\|^2\dd t$, where
    $$
    \aA_k(t):=-\I_{\{s\le \tau^{k}\}}\ppP_N[\alpha(|u_k(t)|^{q}u_k(t)-|v_k(t)|^{q}v_k(t))-\nu\partial_{xx}(u_k(t)-v_k(t))].
    $$
    To this end, let us derive a pathwise estimate for $w_k(t):=u_k(t)-v_k(t)$.   By Proposition~\ref{proposition1} with 
\begin{equation}\label{EEE:eps}
	     \varepsilon:=\frac{a(a\wedge\nu_1)}{4C_1(K+L)} \left(1\wedge \frac{1}{2M}\right)
\end{equation} 
      there exist $T_0>0$ and $N\ge 1$ such that if $T>T_0$, then 
    \begin{align}\label{60}
   & \|w_k(t)\|^2\le C\|w_k(0)\|^2\nonumber\\&\quad\times\exp\left(-at+C_1\varepsilon\int_{kT}^{kT+t}\left(\|u_k\|_{H^1}^2+\|\psi u_k\|_{H^1}^2+\|v_k\|_{H^1}^2+\|\psi v_k\|_{H^1}^2\right)\dd s\right)\nonumber\\&\le C\|w_k(0)\|^2\exp\left(-at+\frac{C_1\varepsilon}{a\wedge \nu_1}\left(\mathcal{E}_{u_k}^{\psi}(t)+\mathcal{E}_{v_k}^{\psi}(t)\right)\right)\nonumber
    \\&\le C\|w_k(0)\|^2   \exp\left(-\frac{a}{2}t+\frac{a}{8(K+L)}\left(2\rho+\|\tilde u(kT)\|^2+\|\tilde u(kT)\|^2\right)\right)
    \end{align}
    for $t\le \tau^k$, where $C>0$ is a constant depending on $a,\nu,\alpha,q,h,\bb_1,\bb_2$. On the other hand, 
    \begin{align}
    \|\tilde u(kT)\|^2&\le \mathcal{E}^{\psi}_{u}(kT)< (K+L)kT+\rho+Md^2, \label{61}\\\label{62}
    \|\tilde u(kT)\|^2&\le \mathcal{E}^{\psi}_{\tilde u}(kT)< (K+L)kT+\rho+Md^2, 
		\end{align}
    and 
    \begin{align}\label{63}
    \|w_k(0)\|^2= \|\tilde u(kT)-\tilde u'(kT)\|^2\le Ce^{-\frac{akT}{2}}e^{C(\rho+d^2)}d^2
		\end{align}
    on the set $\{\sigma\ge kT\}$. Combining \eqref{60}-\eqref{63}, we infer
    \begin{align}\label{64}
    \|w_k(t)\|^2\le Ce^{C(\rho+d^2)}d^2e^{-\frac{at}{2}}e^{-\frac{akT}{4}}
    \end{align}
    for $t\le \tau^k$ on $\{\sigma\ge kT\}$. Moreover, as in \eqref{45d}, we have  
    \begin{align}\label{59}
    \|\aA_k(t)\|^{2}&\lesssim_{N,\nu,\alpha,q} \I_{\{t\le\tau^k\}}[\|w_k\|^2+\|w_k\|^2(\|u_k\|^2_{H^1}+\|v_k\|^2_{H^1})\notag\\&\quad+\|w_k\|^2(\|u_k\|^{\frac{2}{2-q}}+\|v_k\|^{\frac{2}{2-q}})].
    \end{align}
    Due  to \eqref{61} and \eqref{62}, we have
    \begin{align}\label{65}
    \mathcal{E}^{\psi}_{u_k}(t)+\mathcal{E}^{\psi}_{v_k}(t)\le C(t+kT+\rho+d^2)
    \end{align}
    for $t\le \tau^k\wedge\tau^{v_k}$ on $\{\sigma\ge kT\}$. Therefore, by repeating the argument in the derivation of \eqref{45e}-\eqref{45g} combined with the estimates \eqref{64}-\eqref{65}, we obtain
    \begin{align}\label{66}
    \int_0^{\infty}\|\aA_k(t)\|^2\dd t\le Ce^{C(\rho+d^2)}d^2e^{-\frac{akT}{8}}
    \end{align}
    on $\{\sigma\ge kT\}$. Hence, 
    \begin{align*}
    I_1\le \left(\exp\left(Ce^{C(\rho+d^2)}d^2e^{-\frac{akT}{8}}\right)-1\right)^{\frac{1}{2}}.
    \end{align*}
    We set $d<1$ small enough so that
    \[Ce^{C(\rho+d^2)}d^2\le 1.\]
    Then, for $T\ge128/a$,
    \begin{align}\label{67}
    I_1& \le \left(\exp\left(e^{-\frac{akT}{8}}\right)-1\right)^{\frac{1}{2}} \le \sqrt{2}e^{-\frac{akT}{16}}\le e^{-8k+1}\le e^{-2k-5},\quad  k\ge 1.
    \end{align}
    To complete the proof, it remains to bound the terms $I_2, I_3, I_4$ in \eqref{56b}. For $I_2$, we apply the Markov property, Corollary~\ref{corollarya1}, and \eqref{54}:
    \begin{align}\label{68}
        I_2&=\mathbb{E}_{\bm{u}}\left(\I_{\{\sigma\ge kT\}}\mathbb{P}_{\bm{u}(kT)}\{\tau^{{u}}<\infty\}\right)\le \mathbb{E}_{\bm{u}}\left(\mathbb{P}_{\bm{u}(kT)}\{\tau^{{u}}<\infty\}\right)\notag\\&=\mathbb{E}_{\bm{u}}\left(\mathbb{E}_{\bm{u}}\left(\I_{\{kT\le \tau^{{u}}<\infty\}}|\mathscr{F}_{kT}\right)\right)=\mathbb{P}_{\bm{u}}\{kT\le \tau^{{u}}<\infty\}\le 3e^{-\gamma_3 (\rho+LkT)}\notag\\&\le 3e^{-2k-5},
    \end{align}
    where we used the fact that $\gamma_4\le \gamma_3$ and $T\ge1$. The estimate for $I_3$ is the same.
     Finally, by Corollary~\ref{corollaryb1} and \eqref{54}, we have    \begin{align}\label{69}
    I_4\le 3e^{-2k-5}.
    \end{align}
    Combining \eqref{67}-\eqref{69}, we conclude 
    \[\mathbb{P}_{\bm{u}}(\mathcal{Q}_{k}^{''})\le 17e^{-2k-5}\le e^{-2(k+1)},\qquad  k\ge 1.\]
   This completes the proof.
    \end{proof}
    Now we are ready to establish properties \eqref{15}-\eqref{18}. Let us fix $N$ large enough so that Lemma~\ref{lemma3} and Proposition~\ref{proposition1} hold with $\varepsilon$ as in \eqref{EEE:eps}. Then,~\eqref{15} follows from the definition of $\sigma$ and Proposition~\ref{proposition1}. By applying Lemma~\ref{lemma3}, we get
    \[\mathbb{P}_{\bm{u}}\{\sigma=\infty\}\ge 1-\sum_{k=0}^{\infty}\mathbb{P}_{\bm{u}}\{\sigma\in[kT,(k+1)T]\}\ge 1-\frac{2}{e^2-1}>0,\]
    which implies \eqref{16}. Similarly, to check \eqref{17}, we note that
    \begin{align*}
    	\mathbb{E}_{\bm{u}}\left(\I_{\{\sigma<\infty\}}e^{\delta_2\sigma}\right)&\le \sum_{k=0}^{\infty}\mathbb{E}_{\bm{u}}\left(\I_{\{\sigma\in[kT,(k+1)T]\}}e^{\delta_2\sigma}\right)\\&\le 2\sum_{k=0}^{\infty}e^{-2(k+1)}e^{\delta_2(k+1)T}\le \frac{2}{e-1}
    \end{align*}
     for $\delta_2<\frac{1}{T}$. Finally, to show \eqref{18},    from the definition of $\sigma$ we observe that 
      \begin{align*}
    \mathbb{E}_{\bm{u}}\left(\I_{\{\sigma<\infty\}}\left(\|\tilde{u}(\sigma)\|^{4}+\|\tilde{u}'(\sigma)\|^{4}\right)\right)&\le C\mathbb{E}_{\bm{u}}(\I_{\{\sigma<\infty\}}(1+\sigma^{2}))\\&\le C\mathbb{E}_{\bm{u}}\left(\I_{\{\sigma<\infty\}}e^{\delta_2\sigma}\right)\le \frac{2C}{e-1},
       \end{align*} where $C>0$ is a constant depending on $a,\nu,\alpha,q,h,\bb_1,\bb_2$. 
    This completes the proof of exponential squeezing and consequently the proof of Theorem \ref{theorem2}.

\section{Appendix}

 	\subsection{Moment estimate}
	In this subsection, we prove inequality \eqref{19}. By It\^o's formula,  we have 
	$$
		\dd\|u\|^2=2\langle u, h+\nu \partial_{xx} u-\alpha|u|^{q}u-au\rangle  \dd t+ \bb_1\dd t+2\sum_{j=1}^{\infty}b_j\langle u,e_j\rangle \dd\beta_j.
$$
	Taking the expectation, we get
	\[\frac{\dd}{\dd t}\mathbb{E}\|u(t)\|^2 +2a\mathbb{E}\|u(t)\| ^2+2\nu_1\mathbb{E}\|\nabla u(t)\|^2\le 2\mathbb{E}\langle u(t),h\rangle+\bb_1,\]
	which implies 
	\begin{align*}
		\frac{\dd}{\dd t}\mathbb{E}\|u(t)\|^2+a\mathbb{E}\|u(t)\|^2\le \frac{\|h\|^2}{a}+\bb_1.
	\end{align*}
	Applying Gronwall's inequality, we arrive at the required inequality.

  \subsection{Proof of Lemma \ref{lemma1}}\label{proofoflemma31}
  The proof is divided into two steps.
    
    \noindent{\it Step 1.} First, we claim that the set
		\[\chi_A\overline{B}_{H^s}(0,1):=\{\chi_Af\,|\, f\in\overline{B}_{H^s}(0,1) \}\]
		is precompact in $H$. Indeed, let us take any sequence 
		  \[\{\chi_Af_n\}_{n=1}^{\infty}\subset \chi_A\overline{B}_{H^s}(0,1).\]
		 Then, 
		\[\{\chi_Af_n\}_{n=1}^{\infty}\subset H^s(-A,A)\]
		and 
		\[\sup_{n\ge 1}\|\chi_A f_n\|_{H^{s}(-A,A)}=\sup_{n\ge 1}\|\chi_Af_n\|_{H^s}<\infty.\]
		By the compactness of the embedding $H^s(-A,A)\hookrightarrow L^2(-A,A)$, there exists a subsequence $\{\chi_Af_{n_k}\}_{k=1}^{\infty}$ converging in $L^2(-A,A)$. Since $\chi_Af_{n_k}$ is supported in~$(-A,A)$, the sequence $\{\chi_Af_{n_k}\}_{k=1}^{\infty}$ also converges in $H$. This implies the required precompactness property.
		
	    \noindent{\it Step 2.} Let us fix any $\varepsilon, A>0$. It is enough to prove 
	      the inequality \eqref{20} for any~$f\in \overline{B}_{H^s}(0,1)$. Since $\chi_A\overline{B}_{H^s}(0,1)$ is precompact in $H$,   there are 
	    $$
	    \chi_A f_{1}, \ldots,\chi_A f_{m}\in \chi_A\overline{B}_{H^s}(0,1)
	    $$
	     such that 
		\[\chi_A\overline{B}_{H^s}(0,1)\subset \bigcup_{j=1}^m B_{H}(\chi_Af_j,\varepsilon/2).\]
		We choose $N$ large enough so that
		\[\|\qqQ_N \chi_Af_j\|\le \frac{\varepsilon}{2}\]
		for any $j=1,\ldots,m$. Then, for any $\chi_Af \in \chi_A\overline{B}_{H^s}(0,1)$, there is $1\le j\le m$  such~that
		\[\|\chi_A f-\chi_Af_j\|\le \frac{\varepsilon}{2},\]
		which implies 
		\begin{equation*}
			\|\qqQ_N \chi_Af\|\le \|\qqQ_N\chi_A f-\qqQ_N\chi_Af_j\|+\|\qqQ_N \chi_Af_j\|\le\varepsilon.
		\end{equation*}
		 This completes the proof of Lemma \ref{lemma1}.

	\subsection{Weighted estimates for  solutions} 
	
	The main result of this subsection is Proposition~\ref{propositiona1} which establishes a growth estimate for the weighted energy functional $\mathcal{E}^{\psi}_u(t)$ defined in \eqref{13}; therein~$\{u(t)\}_{t\ge0}$ is the solution of \eqref{1} issued from $u\in H$ and $\psi$ is the space-time weight function defined in \eqref{12}. We begin by applying the exponential supermartingale method to estimate the growth of the following two functionals:
  \begin{gather}\label{a1}
  \mathcal{E}_u(t):=\|u(t)\|^2+a\wedge\nu_1\int_0^t\|u(s)\|^2_{H^1}\dd s,\\
\hat{\mathcal{E}}_u^\psi(t):=\|\psi(t)u(t)\|^2+a\wedge\nu_1\int_0^t\left(\|\psi(s)u(s)\|^2+\|\psi(s)\partial_xu(s)\|^2\right)\dd s.\label{a2}
	\end{gather}
	\begin{lemma}\label{lemmaa1}
	There exist   constants $K_1,\gamma_1>0$ depending on $a,\nu,h,\bb_1$ such that 
		\begin{align}\label{a3}
			\mathbb{P}\left\{\sup_{t\ge 0}\left(\mathcal{E}_u(t+T)-K_1t\right)\ge \mathcal{E}_u(T)+\rho\right\}\le e^{-\gamma_1 \rho}
		\end{align}
		for any $\rho>0$,  $T\ge0$, and $u\in H$.
	\end{lemma}
	\begin{proof}
		By It\^o's formula,  we have
		$$
\dd \|u\|^2= 2\langle u,h-\alpha|u|^{q}u-au+\nu\partial_{xx}u \rangle \dd t+ \bb_1 \dd t+2\sum_{j=1}^{\infty}b_j\langle u,e_j\rangle \dd \beta_j,
$$
		which implies
		\begin{align*}
			\mathcal{E}_u(t+T)-\mathcal{E}_u(T)&  +a\wedge\nu_1\int_T^{t+T}\|u(s)\|_{H^1}^2\dd s\\&\le 2\int_{T}^{t+T}\langle u(s),h\rangle  \dd s+\bb_1 t+2\int_T^{t+T}\sum_{j=1}^{\infty}b_j\langle u(s),e_j\rangle \dd\beta_j(s).
		\end{align*}
		Let us consider the martingale 
		\[
		M_1(t):=2\int_T^{t+T}\sum_{j=1}^{\infty}b_j\langle u(s),e_j\rangle \dd\beta_j(s)
		\]whose 
		  quadratic variation   is given by
		\[\langle M_1\rangle(t)=4\sum_{j=1}^{\infty}b_j^2\int_T^{t+T}\langle u(s),e_j\rangle^2\dd s.\]
		By setting $\gamma_1:=\frac{a\wedge\nu_1}{4\bb_1}$ and $K_1:=\frac{2\|h\|^2}{a\wedge\nu_1}+\bb_1$, we infer that
		\begin{align*}
			\mathcal{E}_u(t+T)&-\mathcal{E}_u(T)+a\wedge\nu_1\int_T^{t+T}\|u(s)\|_{H^1}^2 \dd s\\&\le 2\int_{T}^{t+T}\langle u(s),h\rangle \dd s+\bb_1 t+M_1(t)-\frac{\gamma_1}{2} \langle M_1\rangle(t)+\frac{\gamma_1}{2}  \langle M_1\rangle(t),
		\end{align*}
		which leads to 
		\begin{align*}
			\mathcal{E}_u(t+T)-\mathcal{E}_u(T)\le K_1t+M_1(t)-\frac{\gamma_1}{2}\langle M_1\rangle(t).
		\end{align*}
		Hence, for arbitrary $\rho>0$ and $T\ge0$, we have 
		\[\left\{\sup_{t\ge 0}\left(\mathcal{E}_u(t+T)-K_1t\right)-\mathcal{E}_u(T)\ge \rho\right\}\subset \left\{\sup_{t\ge 0}\left(M_1(t)-\frac{\gamma_1}{2} \langle M_1\rangle(t)\right)\ge \rho\right\}.\]
		Applying the exponential supermartingale inequality, we conclude 
		\begin{align*}
			&\mathbb{P}\left\{\sup_{t\ge 0}\left(\mathcal{E}_u(t+T)-K_1t\right)\ge \mathcal{E}_u(T)+\rho\right\}\\&\quad \quad \quad \quad \quad \quad \quad \quad \le\mathbb{P}\left\{\sup_{t\ge 0}\exp\left(\gamma_1 M_1(t)-\frac{\gamma_1^2}{2} \langle M_1\rangle(t)\right)\ge e^{\gamma_1\rho}\right\} \le e^{-\gamma_1\rho}
		\end{align*}
		as required.
	\end{proof}
	\begin{lemma}\label{lemmaa2}  There exist  constants $K_2,\gamma_2,C_2>0$ depending on the parameters $a,\nu,h,\bb_1,\bb_2$ such that  
			\begin{align}\label{a5}
			\mathbb{P}\left\{\sup_{t\ge 0}\left(\hat{\mathcal{E}}_u^\psi(t+T)-K_2t\right)\ge \hat{\mathcal{E}}_u^\psi(T)+C_2\|u(T)\|^2+\rho\right\}\le 2e^{-\gamma_2 \rho}
		\end{align}
for any $\rho>0$,  $T\ge0$, and $u\in H$.	\end{lemma}
	\begin{proof}
		By applying It\^o's formula, we infer that
		\begin{align*}
			\dd\|\psi u \|^2&=2\langle \psi u ,\psi (h-\alpha|u|^{q}u-au+\nu\partial_{xx}u)+\partial_t\psi u \rangle \dd t+\sum_{j=1}^{\infty}b_j^2\|\psi e_j\|^2 \dd t\\&\quad +2\sum_{j=1}^{\infty}b_j\langle \psi u ,\psi e_j\rangle \dd\beta_j ,
		\end{align*}
		which implies 
		\begin{align*}
			\hat{\mathcal{E}}_u^\psi(t+T)-\hat{\mathcal{E}}_u^\psi(T)&+a\wedge\nu_1\int_T^{t+T}\left(\|\psi u \|^2+\|\psi \partial_x u \|^2\right) \dd s\\& \le 2\int_T^{t+T}\left(\langle\psi u ,\psi h+\partial_t\psi u \rangle -2\langle\partial_x\psi \psi u ,\nu\partial_x u \rangle \right) \dd s\\& \quad+\int_T^{t+T}\left(\sum_{j=1}^{\infty}b_j^2\|\psi  e_j\|^2\right)\dd s +2\int_T^{t+T}\sum_{j=1}^{\infty}b_j\langle\psi u ,\psi e_j\rangle \dd\beta_j .
		\end{align*}
		The quadratic variation of the martingale
		\[
		M_2(t):=2\sum_{j=1}^{\infty}\int_T^{t+T}b_j\langle\psi u ,\psi e_j\rangle \dd\beta_j
		\]
	  is given by 
		\[
		\langle M_2\rangle(t)=4\sum_{j=1}^{\infty}b_j^2\int_T^{t+T}\langle\psi u ,\psi e_j\rangle^2\dd s.\]
		Recall that 
		\begin{align}\label{a6}
		\sup_{t\ge 0, x\in \R} \left( |\partial_t \psi(t,x)|+ |\partial_x \psi (t,x)|\right) <\infty
		\end{align}
		and 
		\begin{align}\label{a7}
		0<\psi(t,x)< \varphi(x)\qquad\text{for  $t> 0,x\in\mathbb{R}$}.
		\end{align}
		Then, we infer from \eqref{a6} and \eqref{a7} that
		\begin{align*}
			2\int_T^{t+T}\langle\psi u ,\psi h+\partial_t\psi u \rangle \dd s&\lesssim \int_T^{t+T}\|\psi u \|\left(\|\varphi h\|+\|u \|\right)\dd s,\\
			-4\int_T^{t+T}\langle\partial_x\psi \psi u ,\nu\partial_x u \rangle \dd s&\lesssim |\nu| \int_T^{t+T}\|\psi u \|\|\partial_x u 
		\| \dd s,
		\end{align*}
				and
		\[\langle M_2\rangle(t)\le  4 \bb_2\int_T^{t+T}\|\psi u \|^2\dd s.\]
		By setting  
$\gamma':=\frac{a\wedge\nu_1}{4\bb_2},$
		we see that
		\begin{align}\label{a8}
			\hat{\mathcal{E}}_u^\psi(t+T)-&\hat{\mathcal{E}}_u^\psi(T)\notag\\&\le C_2(a\wedge\nu_1)\int_T^{t+T}\|u\|^2_{H^1}\dd s+\left(\bb_2+C_1\|\varphi h\|^2\right)t \notag\\&\quad+M_2(t)-\frac{\gamma'}{2}\langle M_2\rangle(t)\notag\\&= C_2\left(\mathcal{E}_{u}(t+T)-\mathcal{E}_u(T)\right) +C_2\|u(T)\|^2+\left(\bb_2+C_1\|\varphi h\|^2\right)t\notag\\&\quad+M_2(t)-\frac{\gamma'}{2}\langle M_2\rangle(t),
		\end{align}
		where $C_1,C_2>0$ are    constants depending on $a,\nu$. Let us fix $K_1,\gamma_1$ as   in Lemma~\ref{lemmaa1} and denote
		\[K_2:=\bb_2+C_1\|\varphi h\|^2+C_2K_1.\]
		For arbitrary $\rho>0$, we define the event 
		\begin{align}\label{a9}\Omega^{\rho}:=\left\{\sup_{t\ge 0}(\mathcal{E}_u(t+T)-K_1t)\ge \mathcal{E}_u(T)+\rho\right\}.\end{align}
		Then, 
		\begin{align}\label{a10}
		\mathcal{E}_u(t+T)-\mathcal{E}_u(T)<K_1t+\rho\qquad \text{for any   $t\ge 0$}
		\end{align}
		on $(\Omega^{\rho})^\mathsf{c}$. Hence, in view of \eqref{a8} and \eqref{a10}, we have 
		\begin{gather*}
			\left\{\sup_{t\ge 0}\left(\hat{\mathcal{E}}_u^\psi(t+T)-K_2t\right)-\hat{\mathcal{E}}_u^\psi(T)-C_2\|u(T)\|^2\ge \rho\right\}\mcap \left(\Omega^{\frac{\rho}{2C_2}}\right)^\mathsf{c}\\\subset\left\{\sup_{t\ge 0}\left(M_2(t)-\frac{\gamma'}{2}\langle M_2\rangle(t)\right)\ge \frac{\rho}{2}\right\},
		\end{gather*}
		which implies
		\begin{align*}
			\mathbb{P}\left\{\left\{\sup_{t\ge 0}\left(\hat{\mathcal{E}}_u^\psi(t+T)-K_2t\right)-\hat{\mathcal{E}}_u^\psi(T)-C_2\|u(T)\|^2\ge \rho\right\}\mcap \left(\Omega^{\frac{\rho}{2C_2}}\right)^\mathsf{c}\right\}\le e^{-\frac{\rho\gamma'}{2}}.
		\end{align*}
		Choosing 
		$\gamma_2:= \frac{\gamma'}{2}\wedge\frac{\gamma_1}{2C_2}$, 
        we arrive at 
		\[\mathbb{P}\left\{\sup_{t\ge 0}\left(\hat{\mathcal{E}}_u^\psi(t+T)-K_2t\right)-\hat{\mathcal{E}}_u^\psi(T)-C_2\|u(T)\|^2\ge \rho\right\}\le 2e^{-\gamma_2\rho}\]
		for any $\rho>0$ and $T\ge0$. This completes the proof of the lemma.
	\end{proof}
	Combining Lemmas \ref{lemmaa1} and \ref{lemmaa2}, we obtain a growth estimate for the weighted energy functional $\mathcal{E}_u^{\psi}( t)$.
	\begin{proposition}\label{propositiona1}
		There exist   constants $K_3,\gamma_3,C_3>0$ depending on the parameters $a,\nu,h,\bb_1,\bb_2$ such that 		\begin{align}\label{a11}
			\mathbb{P}\left\{\sup_{t\ge 0}\left( \mathcal{E}_u^{\psi}(t+T)- K_3t\right)\ge  \mathcal{E}_u^{\psi}(T)+ C_3\|u(T)\|^2+\rho\right\}\le 3e^{- \gamma_3 \rho}
		\end{align}
for any $\rho>0$,  $T\ge0$, and $u\in H$.	\end{proposition}
	\begin{proof}
		By   \eqref{a1}, \eqref{a2}, \eqref{13}, we have 
		\begin{align*}
			\mathcal{E}_u^{\psi}(t+T)- \mathcal{E}^{\psi}_u(T)&=\hat{\mathcal{E}}_u^\psi(t+T)-\hat{\mathcal{E}}_u^\psi(T)+\mathcal{E}_u(t+T)-\mathcal{E}_u(T)\notag\\&\quad +a\wedge\nu_1\int_T^{t+T}\|\partial_x \psi u\|^2 \dd s.
		\end{align*}
		In view of \eqref{a6}, there exists a   constant $C_0>0$ depending on $a,\nu$ such that 
		\begin{align}\label{a13}
			\mathcal{E}_u^{\psi}(t+T)- \mathcal{E}^{\psi}_u(T)&\le \hat{\mathcal{E}}_u^\psi(t+T)-\hat{\mathcal{E}}_u^\psi(T)\nonumber\\&\quad+C_0\left(\mathcal{E}_u(t+T)-\mathcal{E}_u(T)\right)+C_0\|u(T)\|^2.
		\end{align}
		Let   $K_1,\gamma_1$ be as   in Lemma~\ref{lemmaa1}, and let $K_2,\gamma_2,C_2$ be as in Lemma~\ref{lemmaa2}. For any~$\rho>0$, let $\Omega^{\rho}$ be the event defined in \eqref{a9}, and let $\bar{\Omega}^{\rho}$ be the event on the left-hand side of \eqref{a5}.
		Then, \eqref{a13} ensures that
		\begin{align}\label{a15}
			\mathcal{E}_u^{\psi}(t+T)- \mathcal{E}^{\psi}_u(T)<  K_3t+\rho+ C_3\|u(T)\|^2\qquad \text{for any   $t\ge 0$}
		\end{align}
		on the event  
		\[\Gamma^{\rho}:=\left(\Omega^{\frac{\rho}{2C_0}}\right)^\mathsf{c}\mcap\left( \bar{\Omega}^{\frac{\rho}{2}}\right)^\mathsf{c},\]
		where
		\[ K_3:=K_2+C_0K_1\qquad  C_3:=C_2+C_0.\]
		Hence,
		\begin{align*}\mathbb{P}\left\{\sup_{t\ge 0}\left( \mathcal{E}_u^{\psi}(t+T)- K_3t\right)- \mathcal{E}^{\psi}_u(T)- C_3\|u(T)\|^2\ge \rho\right\}&\le \mathbb{P}\left((\Gamma^{\rho})^\mathsf{c}\right)\\&\le \mathbb{P}\left(\Omega^{\frac{\rho}{2C_0}}\right)+\mathbb{P}\left(\bar{\Omega}^{\frac{\rho}{2}}\right)
		\\&\le 3e^{- \gamma_3\rho},
		\end{align*}
		where  $\gamma_3:= \frac{\gamma_1}{2C_0}\wedge\frac{\gamma_2}{2}.$
		This completes the proof of the proposition.
	\end{proof}
  As a corollary of the previous proposition, we establish an estimate for the distribution function of the stopping time $\tau^u$ defined by \eqref{13a}, where $u$ is the solution of~\eqref{1} and $K,L, M, \rho$ are non-negative parameters.
	\begin{corollary} \label{corollarya1}
	Let the constants $K_3,\gamma_3,C_3$ be as in Proposition~\ref{propositiona1}. If $K\ge K_3$ and  $M\ge 1+C_3$, then  	
		$$
		\mathbb{P}\{l\le \tau^u<\infty\}\le 3e^{-\gamma_3(\rho+Ll)} 
		$$for any $L,l\ge0, \rho>0,$ and $u\in H$.
	\end{corollary}
	\begin{proof}
		By the definition of $\tau^u$, we have
		\begin{align*}
			\mathcal{E}^{\psi}_u(\tau^u)&\ge K\tau^u+\rho+Ll+M\|u\|^2\\&\ge K_3\tau^u+\rho+Ll+(1+C_3)\|u\|^2
		\end{align*}
		on the event $\{l\le \tau^u<\infty\}$. Therefore, 
		\[\{l\le \tau^u<\infty\}\subset \left\{\sup_{t\ge 0}\left( \mathcal{E}_u^{\psi}(t)- Kt\right)- (1+C_3)\|u\|^2\ge \rho+Ll\right\}.\]
		 Noticing that $\mathcal{E}_u^{\psi}(0)=\|u\|^2$ and applying Proposition~\ref{propositiona1}  with $T=0$, we conclude the proof.
	\end{proof}

	\subsection{Weighted estimate for the truncated process}
	
	Let  $\{u(t)\}_{t\ge0}$ and $\{u'(t)\}_{t\ge0}$ be the solutions of \eqref{1} issued from~$u\in H$ and~$u'\in H$, respectively, and let $\{v(t)\}_{t\ge0}$ be the solution of \eqref{8}. Let us set~$\tau:=\tau^{u}\wedge\tau^{u'}\wedge\tau^{v}$, where $\tau^{u}, \tau^{u'}, \tau^{v}$ are the stopping times given in \eqref{13a} with non-negative parameters $K,L, M, \rho$ to be specified later. In this subsection, we derive a growth estimate for the weighted energy functional $\mathcal{E}_{\hat{u}'}^{\psi}$, where $\{\hat{u}'(t)\}_{t\ge0}$ is a truncated process defined as follows:\footnote{Note that $\{\hat{u}'(t)\}_{t\ge0}$ depends on the parameters $K,L, M, \rho,N$ through the time $\tau$, where~$N$ is the integer in~\eqref{8}. However, let us emphasize that all the estimates in this subsection are uniform in~$N$.} for $t\le \tau$, it coincides with~$\{u'(t)\}_{t\ge0}$, while for~$t\ge \tau$, it solves the equation~\eqref{34}. As a consequence, we obtain an estimate for the distribution function of the stopping time~$\tau^{\hat{u}'}$. 
	\begin{proposition}\label{propositionb1}
		There exist  constants $K_4,\gamma_4,C_4>0$ depending on the parameters $a,\nu,h,\bb_1,\bb_2$ such that  
			\begin{align}\label{b1}
			\mathbb{P}\left\{\sup_{t\ge 0}\left( \mathcal{E}_{\hat{u}'}^{\psi}(t)-K_4t\right)\ge  C_4\|u'\|^2+\rho\right\}\le 3e^{- \gamma_4 \rho}
		\end{align}
		for any $\rho>0$,  $K,L,M\ge0$, and  $u,u'\in H$.
	\end{proposition}
	\begin{proof}
	Consider the auxiliary equation 
   $$
  \partial_t z+az=\nu\partial_{xx} z.
   $$
For any $t,T\ge 0$, we have 
  \begin{align}\label{b3}
   \|z(t+T)\|^2+a\wedge\nu_1\int_T^{t+T}\|z(s)\|^2_{H^1}\dd s\le \|z(T)\|^2.
  \end{align}
  Similarly, using \eqref{a6} and \eqref{b3}, we get
  \begin{align}\label{b4}
   \|\psi(t+T)z(t+T)\|^2+a\wedge\nu_1&\int_{T}^{t+T}\|\psi(s)z(s)\|^2_{H^1}\dd s\notag\\&\lesssim_{a,\nu_1}\int_T^{t+T}\|z(s)\|^2_{H^1} \dd s+\|\psi(T)z(T)\|^2\notag
   \\&\lesssim_{a,\nu_1}\|z(T)\|^2+\|
   \psi(T)z(T)\|^2.
  \end{align}
  Combining \eqref{b3} and \eqref{b4}, we derive 
$$   \mathcal{E}^{\psi}_z(t+T)\lesssim_{a,\nu_1}\mathcal{E}^{\psi}_z(T)\qquad \textup{for $t,T\ge0$}.
$$From this inequality it follows that
  \[\mathcal{E}_{\hat{u}'}^{\psi}(t+\tau)\lesssim_{a,\nu_1}\mathcal{E}_{u'}^{\psi}(\tau)\qquad \textup{for $t\ge0$}\]
  on the event $\{\tau<\infty\}$. Let $K_3,\gamma_3, C_3$ be the   constants in Proposition~\ref{propositiona1}.   Then,  
  \begin{align}\label{b6}
  \mathcal{E}_{\hat{u}'}^{\psi}(t)-C_{a,\nu_1}K_3t\le C_{a,\nu_1}\sup_{t\ge0} \left(\mathcal{E}^{\psi}_{u'}(t)-K_3t\right)\qquad \textup{for $t\ge0$}
  \end{align}
  on $\{\tau<\infty\}$, where $C_{a,\nu_1}>1$ is a constant depending on $a,\nu_1$. Moreover, by the definition of $\hat{u}'$, we see that 
  \begin{align}\label{b7}
  \mathcal{E}_{\hat{u}'}^{\psi}(t)-C_{a,\nu_1}K_3t\le \sup_{t\ge0} \left(\mathcal{E}^{\psi}_{u'}(t)-K_3t\right)\qquad \textup{for $t\ge0$}
  \end{align}
  on $\{\tau=\infty\}$.
  For any $\rho>0$, let
  \[\bar{\bar{\Omega}}^{\rho}:=\left\{\sup_{t\ge 0}\left( \mathcal{E}_{u'}^{\psi}(t)- K_3t\right)\ge   (1+C_3)\|u'\|^2+\rho\right\}.\]
  Combining \eqref{b6} and \eqref{b7}, we derive 
  \[\left(\bar{\bar{\Omega}}^{\frac{\rho}{C_{a,\nu_1}}}\right)^\mathsf{c} \subset \left\{\sup_{t\ge 0}\left( \mathcal{E}_{\hat{u}'}^{\psi}(t)-C_{a,\nu_1} K_3t\right)<  C_{a,\nu_1} (1+C_3)\|u'\|^2+\rho\right\}.\]
  Applying Proposition \ref{propositiona1}, we infer
\begin{align*}
	 \mathbb{P}\left\{\sup_{t\ge 0}\left( \mathcal{E}_{\hat{u}'}^{\psi}(t)-C_{a,\nu_1} K_3t\right)<  C_{a,\nu_1} (1+C_3)\|u'\|^2+\rho\right\}&\ge 1-\mathbb{P}\left(\bar{\bar{\Omega}}^{\frac{\rho}{C_{a,\nu_1}}}\right)\\&\ge 1-3e^{- \frac{\gamma_3}{C_{a,\nu_1}} \rho},
\end{align*}
  which implies \eqref{b1} with $K_4:=C_{a,\nu_1} K_3$,  $\gamma_4:=\frac{\gamma_3}{C_{a,\nu_1}}$, and $C_4:=C_{a,\nu_1} (1+C_3)$.
  \end{proof}
  Now we are able to derive an estimate for the distribution function of $\tau^{\hat{u}'}$, similar to the one obtained in Corollary~\ref{corollarya1}. 
  
  \begin{corollary}\label{corollaryb1}
      If $K\ge K_4$ and  $M\ge C_4$, then  	
      \begin{align}\label{b8}\mathbb{P}\{l\le \tau^{\hat{u}'}<\infty\}\le 3e^{-\gamma_4 (\rho+Ll)} 
      \end{align}for any $L,l\ge0, \rho>0,$ and $u,u'\in H$.
  \end{corollary}
  \begin{proof} Note that
  		\begin{align*}
			\mathcal{E}^{\psi}_{\hat{u}'}(\tau^{\hat{u}'})&\ge K\tau^{\hat{u}'}+\rho+Ll+M\|u'\|^2\\&\ge K_4\tau^{\hat{u}'}+\rho+Ll+C_4\|u'\|^2
		\end{align*}on the event $\{l\le \tau^u<\infty\}$. Hence, 
      \[\{l\le \tau^{\hat{u}'}<\infty\}\subset \left\{\sup_{t\ge 0}\left( \mathcal{E}_{\hat{u}'}^{\psi}(t)-K_4t\right)\ge  C_4\|u'\|^2+\rho+Ll\right\}. \]
      An application of Proposition \ref{propositionb1} yields
      \[\mathbb{P}\{l\le \tau^{\hat{u}'}<\infty\}\le 3e^{-\gamma_4 (\rho+Ll)}\]
      as desired.
  \end{proof}

	 \addcontentsline{toc}{section}{Bibliography}
\newcommand{\etalchar}[1]{$^{#1}$}
\def\cprime{$'$} \def\cprime{$'$}
  \def\polhk#1{\setbox0=\hbox{#1}{\ooalign{\hidewidth
  \lower1.5ex\hbox{`}\hidewidth\crcr\unhbox0}}}
  \def\polhk#1{\setbox0=\hbox{#1}{\ooalign{\hidewidth
  \lower1.5ex\hbox{`}\hidewidth\crcr\unhbox0}}}
  \def\polhk#1{\setbox0=\hbox{#1}{\ooalign{\hidewidth
  \lower1.5ex\hbox{`}\hidewidth\crcr\unhbox0}}} \def\cprime{$'$}
  \def\polhk#1{\setbox0=\hbox{#1}{\ooalign{\hidewidth
  \lower1.5ex\hbox{`}\hidewidth\crcr\unhbox0}}} \def\cprime{$'$}
  \def\cprime{$'$} \def\cprime{$'$} \def\cprime{$'$}
\providecommand{\bysame}{\leavevmode\hbox to3em{\hrulefill}\thinspace}
\providecommand{\MR}{\relax\ifhmode\unskip\space\fi MR }
\providecommand{\MRhref}[2]{%
  \href{http://www.ams.org/mathscinet-getitem?mr=#1}{#2}
}
\providecommand{\href}[2]{#2}

	\bibliographystyle{alpha}

\end{document}